\documentclass[final]{dmtcs-episciences}

\usepackage[utf8]{inputenc}
\usepackage{subfigure}

\usepackage{amsmath}
\usepackage{tikz}

\usepackage{enumerate}
\usepackage[shortlabels]{enumitem}

\usepackage[babel]{microtype}
\usepackage[english]{babel}
\usepackage{hyperref}
\usepackage[capitalise]{cleveref}
\usetikzlibrary{patterns}
\usetikzlibrary{shapes}
\usetikzlibrary{decorations.pathreplacing}
\usetikzlibrary{decorations.pathmorphing}
\usepackage{bbm}
\usepackage{amssymb}
\usepackage{pgf,tikz,tkz-graph}
\usetikzlibrary{arrows,shapes}
\usetikzlibrary{decorations.pathreplacing}
\usepackage{tkz-berge}
\usepackage{enumitem}
\usepackage[normalem]{ulem}
\usepackage[linesnumbered, algoruled, lined]{algorithm2e}
\usepackage{booktabs}

\newcommand*{\floorfrac}[2]{\mathopen{}\left\lfloor\frac{#1}{#2}\right\rfloor\mathclose{}}

\newcommand*{\abs}[1]{\lvert #1\rvert}

\newcommand{\eps}{\varepsilon}

\newcommand{\diam}{diam}

\newcommand{\floor}[1]{\left\lfloor #1 \right\rfloor}
\newcommand{\ceil}[1]{\left\lceil #1 \right\rceil}
\newcommand{\kg}{(k,g)}
\newcommand{\kgd}{(k;g,d)}
\newcommand{\powerSum}[3]{\ensuremath{\sum_{i=#1}^{#2} (#3)^i}}
\newcommand{\degree}[1]{\textrm{deg}({#1})}
\newcommand{\edge}[2]{#1 #2}

\newcommand{\algoName}[1]{\emph{#1}}
\SetKwComment{Comment}{// }{}

\SetCommentSty{mycommfont}
\SetKwIF{If}{ElseIf}{Else}{if}{then}{else if}{else}{}
\SetKwFor{For}{for}{do}{}
\SetKwFor{ForAll}{forall}{do}{}
\SetKwFor{ForEach}{foreach}{do}{}
\SetKwFor{While}{while}{do}{}

\newcommand{\numCombsKgdGenerated}{177}
\newcommand{\numOrderGenDeterminedNew}{107}
\newcommand{\numExhGen}{123}

\newtheorem{theorem}{Theorem}
\newtheorem{proposition}{Proposition}
\newtheorem{corollary}{Corollary}
\newtheorem{lemma}{Lemma}
\newtheorem{claim}{Claim}
\newtheorem{question}{Question}
\newtheorem{definition}{Definition}

\usepackage{natbib}

\author[S.\ Cambie et al.]{Stijn Cambie\affiliationmark{1}\thanks{Supported by a Postdoctoral Fellowship of the Research Foundation Flanders (FWO) with grant number 1225224N.}
  \and Jan Goedgebeur\affiliationmark{1,2}\thanks{Supported by Internal Funds of KU Leuven and a grant of the Research Foundation Flanders (FWO) with grant number G0AGX24N.}\\
  \and Jorik Jooken\affiliationmark{1}\thanks{Supported by a Postdoctoral Fellowship of the Research Foundation Flanders (FWO) with grant number 1222524N.}
    \and Tibo Van den Eede\affiliationmark{1}\textsuperscript{\textdagger}}
\title{On the order-diameter ratio of girth-diameter cages}
\affiliation{
  Department of Computer Science, KU Leuven Campus Kulak-Kortrijk, Kortrijk, Belgium\\
  Department of Mathematics, Computer Science and Statistics, Ghent~University, Ghent, Belgium}
\begin{document}
% This is only used if you are compiling for a volume before vol 25
% \publicationdetails{VOL}{2015}{ISS}{NUM}{SUBM}
% This is the new form of collecting the data, starting with vol 25

\publicationdata{vol. 28:4, SOFSEM 2026}{2026}{1}{10.46298/dmtcs.17595}{2026-02-27; None}{2026-06-20}

\maketitle
\begin{abstract}
 For integers $k,g,d$, a $(k;g,d)$-cage (or simply girth-diameter cage) is a smallest $k$-regular graph of girth $g$ and diameter $d$ (if it exists). The order of a $(k;g,d)$-cage is denoted by $n(k;g,d)$. We determine asymptotic lower and upper bounds for the ratio between the order and the diameter of girth-diameter cages as the diameter goes to infinity. We also prove that this ratio can be computed in constant time for fixed $k$ and $g$.

 We theoretically determine the exact values $n(3;g,d)$, and count the number of corresponding girth-diameter cages, for $g \in \{4,5\}$.
 Moreover, we design and implement an exhaustive graph generation algorithm and use it to determine the exact order of several open cases and obtain -- often exhaustive -- sets of the corresponding girth-diameter cages. The largest case we generated and settled with our algorithm is a $(3;7,35)$-cage of order 136. 
\end{abstract}

\section{Introduction}

In this paper, we study girth-diameter cages, a concept recently introduced by Araujo-Pardo et al.~\citep{ACGKL25}, which is closely related to the Cage Problem and also has connections to the Degree Diameter Problem. These are two important problems in the area of extremal graph theory. We first introduce these two problems before discussing girth-diameter cages.

Given integers $k \geq 2$, $g \geq 3$, the \textit{Cage Problem} asks for the minimum order of a $k$-regular graph of girth $g$ (denoted by $n(k,g)$) and the corresponding graphs (called \textit{$(k,g)$-cages}). 
This problem started receiving a lot of attention after Sachs~\citep{sachs1963regular} proved that $k$-regular graphs with girth $g$ (called \textit{$(k,g)$-graphs}) exist for every value of $k \geq 2$ and $g \geq 3$. 
The exact value of $n(k,g)$ has only been determined for a few infinite families and for some specific small pairs $(k,g)$ (mainly using clever algorithmic searches~\citep{EJ08}), but remains wide open for many pairs $(k,g)$. Apart from some exceptional cases, the best upper bounds on $n(k,g)$ are given by an explicit construction due to Lazebnik, Ustimenko and Woldar~\citep{LUW97}. The state of the art as such tells us that $M(k,g) \le n(k,g) \le C \cdot M(k,g)^{3/2} $ for some constant $C$, where $M(k,g)$ denotes the well-known Moore bound (also presented later in~\eqref{eq:Mbound}) for $k$-regular graphs of girth $g$. 

The \text{Degree Diameter Problem} asks for the order of the largest graph of given maximum degree $\Delta$ and diameter $d$ (denoted by $n'(\Delta,d)$) as well as the corresponding graphs. The Cage Problem and the Degree Diameter Problem are connected through the Moore bound, which simultaneously gives a lower bound for $n(k,g)$ as well as an upper bound for $n'(\Delta,d)$.

For more detailed information, we refer the interested reader to the Dynamic Cage Survey by Exoo and Jajcay~\citep{EJ08} on the Cage Problem and to the excellent survey by Miller and \v{S}ir\'a\v{n}~\citep{MS12} on the Degree Diameter Problem.

A \textit{$(k;g,d)$-graph} is a $k$-regular graph of girth $g$ and diameter exactly $d$. Analogously to the classical cage problem, a smallest $(k;g,d)$-graph is called a \textit{$(k;g,d)$-cage} (or simply a \textit{girth-diameter cage} if the values of $k, g$ and $d$ are not specified) and its order is denoted by $n(k;g,d)$. The authors of~\citep{ACGKL25} mainly studied the case $n(k;5,4)$. We will be interested in the exact determination of $n(k;g,d)$ for certain specific tuples $(k,g,d)$ as well as determining the smallest real number $f(k,g)$ such that, as the diameter $d$ goes to infinity, we have $n(k;g,d) \leq f(k,g)d+O_{k,g}(1).$ Here, $O_{k,g}(1)$ is a constant depending only on $k$ and $g$. In cases where the exact determination of $f(k,g)$ or $n(k;g,d)$ is not possible, we will be interested in lower and upper bounds instead.

Our main theorem can be stated as follows and is proven in~\cref{sec:main}.

\begin{theorem}\label{thm:general_estimate}
    For all integers $k,g \geq 3$, 
    we have $\frac{M(k,g)}{g} \leq f(k,g)$ and $f(k,g) \leq \frac{n(k,g)}{g}.$ Moreover, neither of the two bounds are sharp for all pairs $(k,g)$.
\end{theorem}

The exact value of $n(k;g,d)$ is open for most cases except for some notable examples from the literature that were already known (and are implied by our result). When the girth equals $4$, using $n(k,4)=M(k,4)=2k$,~\cref{thm:general_estimate} implies that $n(k;4,d) = \frac k 2d+O_k(1)$, which is already known by~\cite[Thm.~2]{EPPT89} (it is stated for $K_3$-free graphs).
Also without girth condition, up to the $O_k(1)$ term, sharpness of the ratio $\frac{M(k, 3)}{3}=\frac{k+1}{3}$ was already known by classical results in~\citep{EPPT89}, and made precise by~\citep{Knor2014}. 

Note that if a $(k,g)$-Moore graph exists, i.e., $M(k,g)=n(k,g)$, then $n(k; g,d)$ is determined up to a constant by~\cref{thm:general_estimate}.
This is the case when $g \in \{3,4\}$, $g=5$ and $k \in\{ 2, 3, 7\}$ and possibly $k=57$ (the existence of such a Moore graph is a famous open problem~\citep{EJ08}), and $g \in \{6,8,12\}$ for some values of $k$ (e.g.\ when $k-1$ is a prime power).
So the smallest cases (for $k$ resp.\ $g$) for which the bounds are not asymptotically determined, are $(k,g)\in\{(3,7),(4,5)\}$. We will study these pairs in~\cref{sec:determined_bounds}.

We also prove in \cref{sec:thmFiniteTask} that for all integers $k,g \geq 3$, determining $f(k,g)$ is a task that requires only finitely many computations\footnote{We remark that this is mainly of theoretical importance, since for many $(k,g)$ this finite number is extremely large.}. 

\begin{theorem}\label{thm:finiteTask}
    For all integers $k,g \geq 3$, one can compute $f(k,g)$ in $O_{k,g}(1)$ time.
\end{theorem}

In~\cref{sec:determined_bounds}, we compute several previously unknown values of $n(k;g,d)$, starting with the theoretical determination of $n(3;4,d)$ and $n(3;5,d)$, as well as the number of corresponding girth-diameter cages.
Additionally, we design and implement an exhaustive generation algorithm for $(k;g,d)$-graphs. This allowed us to determine the values of $n(k;g,d)$ for $\numOrderGenDeterminedNew$ new triples $(k,g,d)$ and the corresponding $(k;g,d)$-cages (often exhaustive).
The results derived from the computations are listed in Appendix~\ref{sec:tables_generator}.

Finally, in~\cref{sec:conclusion} we conclude with some further observations and open questions, in particular a question relating $(k;g,d)$-cages to the conjecture that every even girth $(k,g)$-cage is bipartite.

\section{Proof of~\cref{thm:general_estimate}}\label{sec:main}

In this section, we prove \cref{thm:general_estimate} in three parts separately. First, we prove a lower bound $M'\kgd$ on $n\kgd$ in \cref{prop:imprLowerBound} which implies that $\frac{M(k,g)}{g} \leq f(k,g)$. As a side result, we also prove a better lower bound $M''\kgd$ for $d \leq g$ where $g$ is even, i.e., $g=2t \geq 4$ for some $t$, in \cref{prop:imprLowerBound_specialimprovement}. The proofs of these lower bounds contain the underlying ideas for the generation algorithm described in Subsection~\ref{sec:gen_algo}. Second, we prove $ f(k,g) \leq \frac{n(k,g)}{g}$ in \cref{prop:upp_bnd}. Third, we prove that neither of the two bounds on $f(k,g)$ are sharp for all pairs $(k,g)$ in \cref{prop:main_not_sharp}.

Recall the Moore bound $M\kg$, defined by
\begin{equation}\label{eq:Mbound}
	n\kg \geq M\kg =
	\begin{cases}
		1 + k \powerSum{0}{t-1}{k-1} & \text{for } g=2t+1 \\
		%\noalign{\vskip9pt}
		2 \powerSum{0}{t-1}{k-1} & \text{for } g=2t . \\
	\end{cases}
\end{equation}

This bound is obtained by realising that for any vertex $u$ in a $(k,g)$-graph, there cannot be an edge between two vertices $v$ and $w$ if the sum of the distance between $u$ and $v$ and between $u$ and $w$ is strictly smaller than $g-1$. Hence, locally around every vertex $u$ in a $(k,g)$-graph, the graph looks like a tree of order $M(k,g)$ (and we call this tree a Moore tree). We will furthermore split $M(k,2t+1)$ in the sum of the sizes of even and odd neighbourhoods towards the center of the Moore tree;
$M(k,2t+1)=M_0(k,2t+1)+M_1(k,2t+1)$, where
$$M_0(k,2t+1)= 1+k\sum_{i=0}^{\floorfrac{t-2}{2}} (k-1)^{2i+1} \mbox{ and }
M_1(k,2t+1)= k \sum_{i=0}^{\floorfrac{t-1}{2}}(k-1)^{2i}.$$

Note that it is usually only defined for $k\geq2$ and $g\geq3$, but we allow the formula to be interpreted for any $g > 0$ and further define $M(k,g)=0$ for every $g \le 0.$ 

In~\cref{prop:imprLowerBound}, we give a lower bound for $n\kgd$, which we denote by $M'\kgd$ in analogy to this.
For certain values, this lower bound will turn out to be exact (see~\cref{sec:determined_bounds}). A direct corollary of this theorem is that $f(k,g) \geq \frac{M(k,g)}{g}$.

\begin{proposition}\label{prop:imprLowerBound}
		Suppose $k \geq 3, g \geq 3, d \geq t:= \floor{g/2}$,
		then $n\kgd \geq M'\kgd$, where
		\begin{equation*}
			 M'\kgd = 
			\begin{cases}
				M(k,g) + M(k,2d-2t-1) & \text{for } t \leq d \le 2t, \\
				\noalign{\vskip9pt}
				(r+2)M(k,g)+M(k,s) & \text{for } 
				\vcenter{\hbox{\shortstack[c]{
							$d=2t+1+rg+s \text{ with }$\\
							$r\geq0,\; s\in\{0,\ldots,g-1\}$
				}}}
			\end{cases}
		\end{equation*}
	\end{proposition}
	%\newpage
    	\begin{proof}
        Let $G$ be a $(k;g,d)$-graph of order $n=n(k;g,d)$ and let $u, v$ be two vertices for which $d(u,v)=d=\diam (G).$
    
        We define the neighbourhoods $N_i := N_i(u)= \{w \in V(G)~|~d(u,w)=i\}$. It holds that $|V(G)| = \sum_{i=0}^{d} |N_i|$. 
    	Since $G$ is a $\kg$-graph, we claim that 
        \begin{equation}\label{eq:Moore}
            \sum_{i=0}^{t} |N_i| \geq M(k,g).
        \end{equation}
    
        To see this, note that if $g$ is odd, $\cup_{ 0\le i \le t} N_i$ spans a (supergraph of a) $k$-ary tree of order $M(k,g),$ as one can see by taking a rooted tree starting from $u$ with height $t$, i.e., by considering a $(k,g)$-Moore tree rooted at the vertex $u$.
        If $g$ is even, the claim follows by considering $\cup_{ 1\le i \le t-1} \left( N_i(u) \cup N_i(z) \right) \subseteq \cup_{ 0\le i \le t} N_i(u)$ for a neighbour $z$ of $u$, i.e., by considering a $(k,g)$-Moore tree rooted at the edge $uz$.

        We now deal with the case $t \leq d \le 2t$. Similarly, by considering a rooted subtree in $v$ of height $d-t-1$, one concludes that $$\sum_{i=0}^{d-t-1} |N_{d-i}| \geq M(k,2(d-t)-1).$$  
        Hence, for $t \leq d \le 2t$, we obtain $|V(G)| \geq M(k,g) + M(k,2(d-t)-1)$.
        
    	Henceforth, we focus on the other case. Let $G'$ be a $\kgd$-graph with $d=2t+1+rg+s$ where $r\geq0, s \in \{0,\ldots,g-1\}$. For every $1 \leq j \leq g$ and for every $0 \leq q \leq d-j$, similarly as before, we conclude that $\sum_{i=1}^{j} |N_{q+i}| \geq M(k,j)$, by considering rooted subtree(s) from a vertex $w \in N_{q+\ceil{\frac{j}{2}}}$ if $j$ is odd, or edge $xy \in N_{q+\frac{j}{2}} \times N_{q+\frac{j}{2}+1}$ if $j$ is even.
        Hence, the theorem follows since
        \begin{equation*}
            \begin{split}
                |V(G')| & =  \sum_{i=0}^{t} |N_i| + \sum_{l=0}^{r-1} \sum_{i=1}^{g} |N_{t+lg+i}|  + \sum_{i=1}^{s} |N_{t+rg+i}|  + \sum_{i=0}^{t} |N_{d-i}| \\		
                & \geq M(k,g) + r M(k,g) + M(k,s) + M(k,g) 	\\
                & = (r+2)M(k,g)+M(k,s) . 
            \end{split} 
        \end{equation*}
    \end{proof}

When the girth $g$ is even, and $d \le g$, the above bounds can be sharpened.

\begin{proposition}\label{prop:imprLowerBound_specialimprovement}
		Suppose $k \geq 3$ and $d \le g=2t \geq 4$, then $n\kgd \geq M''\kgd $, where
        $$M''\kgd =M(k,g)+2 \max\{M_0(k, 2d-2t-1),M_1(k, 2d-2t-1) \}.$$
		Furthermore, this bound can only be attained by bipartite graphs.
	\end{proposition}
    \begin{proof}
        Let $u,v$ be vertices satisfying $d(u,v)=d.$
        Since $ \cup_{i=0}^{t-1} N_i(u)$ and $ \cup_{i=0}^{t-1} N_i(v)$
        span a tree, we can deduce that in each case (depending on $d \pmod 2$ and $t \pmod 2$) the following neighbourhoods form an independent set (and so does their union, since distances between vertices of different neighbourhoods are at least two)
        % $$\begin{cases}
        \begin{align*}
            \left(N_1(u) \cup N_3(u) \cup \ldots \cup N_{t-1}(u) \right) \cup \left(N_1(v) \cup N_3(v) \cup \ldots \cup N_{d-t-1}(v) \right),\\
                \left(N_1(u) \cup N_3(u) \cup \ldots \cup N_{t-1}(u) \right) \cup \left(N_0(v) \cup N_2(v) \cup \ldots \cup N_{d-t-1}(v) \right),\\
				\left(N_0(u) \cup N_2(u) \cup \ldots \cup N_{t-1}(u) \right) \cup \left(N_0(v) \cup N_2(v) \cup \ldots \cup N_{d-t-1}(v) \right),\\
                \left(N_0(u) \cup N_2(u) \cup \ldots \cup N_{t-1}(u) \right) \cup \left(N_1(v) \cup N_3(v) \cup \ldots \cup N_{d-t-1}(v) \right).
        \end{align*}
	% \end{cases}$$
    
    Since the graph needs to be $k$-regular, its order is at least two times the independence number (and equality implies the bipartiteness), which is at least the sum of sizes of the corresponding neighbourhoods listed above. Here $M(k,g)=2M_\eps(k, 2t-1)$ for $\eps \in \{0,1\}$ satisfying $ 2 \nmid t+\eps$ ($\eps \equiv t-1 \pmod 2$). This can be concluded from the fact that if $M(k,g)$ would be attained, the extremal graph is bipartite and regular, but also from direct verification where one uses $k=1+(k-1).$
\end{proof}

\begin{corollary}
    Suppose $k \geq 3, g \geq 3, d \geq t:= \floor{g/2}$, then $n\kgd \geq M\kgd$, where
    \begin{equation*}
			M\kgd = 
			\begin{cases}
				M''\kgd & \text{for } d \le g=2t, \\
                M'\kgd & \text{otherwise.}
			\end{cases}
		\end{equation*}
\end{corollary}

Note here that $M''\kgd \geq M'\kgd$ for all $t \le d \le g=2t \geq 4$, since $M(k,2d-2t-1)=M_0(k,2d-2t-1) + M_1(k,2d-2t-1) \le 2 \max\{ M_0(k,2d-2t-1), M_1(k,2d-2t-1)\}.$

We now discuss a construction leading to an infinite family of $(k,g)$-graphs for which the ratio between the order and the diameter tends to $\frac{n(k,g)}{g}$ as the diameter of the graph tends to infinity.

\begin{proposition}\label{prop:upp_bnd}
    For all integers $k,g \geq 3$, 
    we have $ f(k,g) \leq \frac{n(k,g)}{g}.$
\end{proposition}

\begin{proof}
    Let $M$ be a $(k,g)$-cage of order $n(k,g).$
    
    If $k$ is even, let $M'$ be the graph obtained by removing a vertex $u$ (disjoint from at least one shortest cycle of $M$), adding two vertices $v$ and $w$ and adding edges between $v$ and $2$ of the initial neighbours of $u$ and edges between $w$ and the $k-2$ remaining initial neighbours of $u$. Let $r \geq 1$ be an integer and consider the graph $M'_r$ obtained by taking $r$ copies of $M'$ (where vertices in copy $i$ receive the subscript $i$) and identifying $v_i$ with $w_{i-1}$ for each $2 \leq i \leq r$. Note that $|V(M'_r)|=r|V(M)|+1$, that $M'_r$ has girth $g$, that the distance between $v_1$ and $w_r$ is at least $gr$ and that every vertex in $M'_r$ has degree $k$, except for $v_1$ (having degree $2$) and $w_r$ (having degree $k-2$). Let $G'$ be the graph obtained by taking a $(k,g+2)$-cage, removing a vertex $x$, removing an edge $yz$ where $y$ was initially a neighbour of $x$ and adding a matching of size $k/2$ between the vertices of degree $k-1$. Note that $G'$ has girth at least $g$ and all vertices have degree $k$, except for $y$ (having degree $k-2$). We can complete $M'_r$ to a $(k,g)$-graph $G''$ by identifying $y$ with $v_1$ and adding another copy of $M$ with one edge subdivided, and identifying this additional vertex (of degree $2$) with $w_{r}$. Now $G''$ is $k$-regular, the diameter of $G''$ is at least as large as the diameter of $M'_r$, $G''$ has girth $g$ and for $r$ tending to infinity, the ratio of its order and diameter tends to $\frac{n(k,g)}{g}.$

    If $k$ is odd, let $M'$ be the graph obtained by removing an edge $vw$ (disjoint from at least one shortest cycle) from $M$. Let $r \geq 1$ be an integer and consider the graph $M'_r$ obtained by taking $r$ copies of $M'$ and adding an edge between $v_i$ and $w_{i-1}$ for each $2 \leq i \leq r$. Now $|V(M'_r)|=r|V(M)|$, $M'_r$ has girth equal to $g$, the distance between $v_1$ and $w_r$ is at least $gr-1$ and each vertex in $M'_r$ has degree $k$, except for $v_1$ and $w_r$, which each have degree $k-1$. Let $G'$ be the graph obtained by taking a $(k,g+1)$-cage, removing a vertex $x$ and adding a matching of size $\frac{k-1}{2}$ between $k-1$ of the initial neighbours of $x$ (so $G'$ has girth at least $g$ and all vertices, except for an initial neighbour $y$ of $x$, have degree $k$). Let $G''$ be the graph obtained by taking $M'_r$, taking two copies $G_1'$ and $G_2'$ of $G'$, adding an edge between $y_1$ and $v_1$ and an edge between $y_2$ and $w_r$. Again, $G''$ is a $k$-regular graph with girth $g$ for which the ratio between its order and diameter tends to $\frac{n(k,g)}{g}$ as $r$ tends to infinity. An illustration of this is presented in~\cref{fig:Prop3Illus}.
    Here the matching of order $(3-1)/2=1$ in $G'$ is the dashed line, and dotted lines are used for the edge connecting $M'_r$ and $G'$, as well as the edges $v_iw_{i-1}.$
\end{proof}

    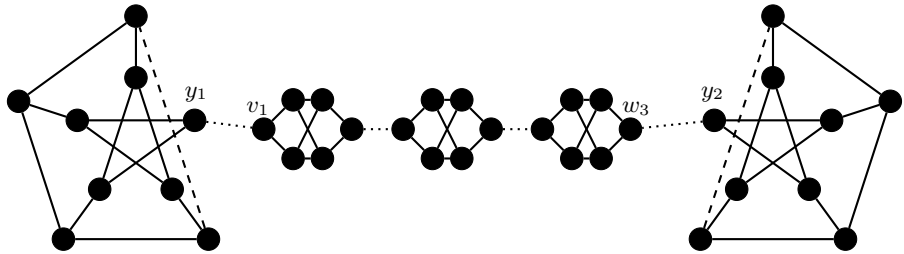
\begin{figure}[h]
    \centering
\begin{tikzpicture}[scale=0.39, main_node/.style={fill,draw,minimum size=0.21,circle,inner sep=3pt]}]
% \draw[fill] (\x,\y) circle (0.25);

\node[main_node] (0) at (-0.8571428571428563, 4.857142857142858) {};
\node[main_node] (1) at (-4.857142857142858, 1.9509727452415901) {};
% \node[main_node] (2) at (3.1428571428571423, 1.9509727452415901) {};
\node[main_node] (3) at (-0.8571428571428563, 2.7542184086770387) {};
\node[main_node] (4) at (-3.329278812161496, -2.751309273284966) {};
\node[main_node] (5) at (-2.857142857142856, 1.3011333527264064) {};
\node[main_node] (6) at (1.6149930978757807, -2.751309273284966) {};
\node[main_node] (7) at (0.3789251203664623, -1.050007656536872) {};
\node[main_node] (8) at (-2.093210834652175, -1.050007656536872) {};
\node[main_node] (9) at (1.1428571428571408, 1.3011333527264064) {};

\node (40) at (18,1){};

\node () at (1.2,2.20) {$y_{1}$};
\node () at (18.8,2.20) {$y_{2}$};
\node () at (3.3,1.70) {$v_{1}$};
\node () at (16.2,1.70) {$w_{3}$};
\node[main_node] (14) at (3.5,1){};
\node[main_node] (15) at (4.5,0){};
\node[main_node] (16) at (4.5,2){};
\node[main_node] (18) at (5.5,2){};
\node[main_node] (17) at (5.5,0){};
\node[main_node] (19) at (6.5,1){};

\node[main_node] (24) at (8.25,1){};
\node[main_node] (25) at (9.25,0){};
\node[main_node] (26) at (9.25,2){};
\node[main_node] (28) at (10.25,2){};
\node[main_node] (27) at (10.25,0){};
\node[main_node] (29) at (11.25,1){};

\node[main_node] (34) at (13,1){};
\node[main_node] (35) at (14,0){};
\node[main_node] (36) at (14,2){};
\node[main_node] (38) at (15,2){};
\node[main_node] (37) at (15,0){};
\node[main_node] (39) at (16,1){};

 \path[draw, thick]
% (0) edge node {} (1) 
(6) edge[dashed] node {} (0) 
(9) edge[dotted] node {} (14) 
(19) edge[dotted] node {} (24)
(29) edge[dotted] node {} (34)
%(39) edge[dotted] node {} (40)

\foreach \i/\j in {14/15,14/16,15/17,16/17,15/18,16/18,17/19,18/19}{(\i) edge node {} (\j) }

\foreach \i/\j in {24/25,24/26,25/27,26/27,25/28,26/28,27/29,28/29}{(\i) edge node {} (\j) }

\foreach \i/\j in {34/35,34/36,35/37,36/37,35/38,36/38,37/39,38/39}{(\i) edge node {} (\j) }

(0) edge node {} (1) 
% (0) edge node {} (2) 
(0) edge node {} (3) 
(1) edge node {} (4) 
(1) edge node {} (5) 

% (2) edge node {} (6) 
% (2) edge node {} (9) 
(3) edge node {} (7) 
(3) edge node {} (8) 
(4) edge node {} (6) 
(4) edge node {} (8) 
(5) edge node {} (7) 
(5) edge node {} (9) 
(6) edge node {} (7) 
(8) edge node {} (9) 
;

\node[main_node] (50) at (20.857142857142858, 4.857142857142858) {};
\node[main_node] (51) at (24.857142857142858, 1.9509727452415901) {};
\node[main_node] (53) at (20.857142857142858, 2.7542184086770387) {};
\node[main_node] (54) at (23.329278812161498, -2.751309273284966) {};
\node[main_node] (55) at (22.857142857142854, 1.3011333527264064) {};
\node[main_node] (56) at (18.385006902124218, -2.751309273284966) {};
\node[main_node] (57) at (19.62107487963354, -1.050007656536872) {};
\node[main_node] (58) at (22.093210834652176, -1.050007656536872) {};
\node[main_node] (59) at (18.857142857142858, 1.3011333527264064) {};

 \path[draw, thick]
(56) edge[dashed] node {} (50) 
(59) edge[dotted] node {} (39)

(50) edge node {} (51) 
% (0) edge node {} (2) 
(50) edge node {} (53) 
(51) edge node {} (54) 
(51) edge node {} (55) 

% (2) edge node {} (6) 
% (2) edge node {} (9) 
(53) edge node {} (57) 
(53) edge node {} (58) 
(54) edge node {} (56) 
(54) edge node {} (58) 
(55) edge node {} (57) 
(55) edge node {} (59) 
(56) edge node {} (57) 
(58) edge node {} (59) 
;

\end{tikzpicture}
\caption{Illustration of the construction in~\cref{prop:upp_bnd} for $(k,g)=(3,4)$ and $r=3$.}\label{fig:Prop3Illus}
\end{figure}

Next, it is natural to wonder whether for all integers $k,g \geq 3$, we have $f(k,g)=\frac{M(k,g)}{g}$ or $f(k,g)=\frac{n(k,g)}{g}$. We will show that this is not the case.
\begin{proposition}\label{prop:main_not_sharp}
    Neither of the two bounds $\frac{M(k,g)}{g} \leq f(k,g)$ and $f(k,g) \leq \frac{n(k,g)}{g}$ are sharp for all pairs $(k,g)$.
\end{proposition}

\begin{proof}
    Let $G$ be a large enough (to be determined later) $(3;7,d)$-cage, let $u,v \in V(G)$ be such that $d(u,v)=d=\diam(G)$ and define $N_i := \{w \in V(G)~|~d(u,w)=i\}$. If $f(3,7)=\frac{M(3,7)}{7}=\frac{22}{7},$ then there exists an integer $c_1$ such that $|N_i|+|N_{i+1}|+\ldots+|N_{i+6}|=22$ for all $i \in \{0,1,\ldots,d-6\}$ except for at most $c_1$ choices of $i$. Therefore, we may assume that there exists some $i \in \{0,1,\ldots,d-13\}$ such that $|N_i|+|N_{i+1}|+\ldots+|N_{i+6}|=22$, $|N_{i+1}|+|N_{i+2}|+\ldots+|N_{i+7}|=22$, $\ldots$, $|N_{i+7}|+|N_{i+7+1}|+\ldots+|N_{i+7+6}|=22$ (by choosing $G$ large enough in terms of $c_1$). Since the sums are all equal to $22$, we have for every $j \in \{0,\ldots,6\}$ that $|N_{i+j}|=|N_{i+j+7}|$, i.e.\ the neighbourhood sizes are repetitive. By the pigeonhole principle, there exists some $j \in \{0,\ldots,6\}$ such that $|N_{i+j}| \leq 3$. If $|N_{i+j}|=1$, then $|N_{i+j+7}|=1$. Let $G'$ be the graph induced by $N_{i+j} \cup \ldots \cup N_{i+j+7}$. By identifying $u \in N_{i+j}$ with $v \in N_{i+j+7}$, we would obtain a $(3,7)$-graph of order $22$, which is impossible since $n(3,7)=24>22$.
    If $|N_{i+j}|=2$ or $|N_{i+j}|=3$, then we checked that this substructure is again impossible, by computer verification (more precisely, by exhaustively generating graphs) based on this condition and the $7$ next neighbourhoods. Details of the computer verification can be found in Appendix~\ref{app:comp_ver}.
    So we conclude that $f(3,7)>\frac{22}{7}$.

    For $(k,g)=(4,7),$ there is a graph $G$ with girth $7$, order $76$, with all vertices having degree $4$ except for $2$ which have degree $2$ and are distance $8$ apart (see~\cref{fig:almost4RegGirth7}, it is derived from the data set from~\citep{EJJZ25}).
     This graph is made publicly available on the House of Graphs~\citep{HOG}, as the graph with HoG-id $54041$\footnote{It can be downloaded from \url{https://houseofgraphs.org/graphs/54041}}.
     By letting $M'=G$ in the proof of \cref{prop:upp_bnd} and using the construction described therein, we obtain $f(4,7) \leq \frac{75}{8}<\frac{n(4,7)}{7}=\frac{67}{7}$.
\begin{figure}[h]
    \centering

    \begin{tikzpicture}[scale=0.45]
\draw[fill] (18.00,4.00) circle (0.15);
\draw[fill] (12.00,0.61) circle (0.15);
\draw[fill] (6.00,2.00) circle (0.15);
\draw[fill] (12.00,1.22) circle (0.15);
\draw[fill] (6.00,8.00) circle (0.15);
\draw[fill] (12.00,10.96) circle (0.15);
\draw[fill] (12.00,7.91) circle (0.15);
\draw[fill] (12.00,12.17) circle (0.15);
\draw[fill] (18.00,2.00) circle (0.15);
\draw[fill] (12.00,5.48) circle (0.15);
\draw[fill] (6.00,10.00) circle (0.15);
\draw[fill] (12.00,12.78) circle (0.15);
\draw[fill] (18.00,6.00) circle (0.15);
\draw[fill] (12.00,6.09) circle (0.15);
\draw[fill] (12.00,1.83) circle (0.15);
\draw[fill] (12.00,3.65) circle (0.15);
\draw[fill] (6.00,4.00) circle (0.15);
\draw[fill] (12.00,9.74) circle (0.15);
\draw[fill] (12.00,6.70) circle (0.15);
\draw[fill] (18.00,12.00) circle (0.15);
\draw[fill] (12.00,10.35) circle (0.15);
\draw[fill] (6.00,12.00) circle (0.15);
\draw[fill] (12.00,8.52) circle (0.15);
\draw[fill] (6.00,6.00) circle (0.15);
\draw[fill] (12.00,2.43) circle (0.15);
\draw[fill] (12.00,4.26) circle (0.15);
\draw[fill] (18.00,10.00) circle (0.15);
\draw[fill] (12.00,11.57) circle (0.15);
\draw[fill] (12.00,9.13) circle (0.15);
\draw[fill] (12.00,7.30) circle (0.15);
\draw[fill] (12.00,13.39) circle (0.15);
\draw[fill] (18.00,8.00) circle (0.15);
\draw[fill] (12.00,4.87) circle (0.15);
\draw[fill] (12.00,3.04) circle (0.15);
\draw[fill] (9.00,0.74) circle (0.15);
\draw[fill] (9.00,9.58) circle (0.15);
\draw[fill] (15.00,12.53) circle (0.15);
\draw[fill] (15.00,13.26) circle (0.15);
\draw[fill] (9.00,7.37) circle (0.15);
\draw[fill] (15.00,3.68) circle (0.15);
\draw[fill] (9.00,5.16) circle (0.15);
\draw[fill] (15.00,11.05) circle (0.15);
\draw[fill] (15.00,0.74) circle (0.15);
\draw[fill] (9.00,2.95) circle (0.15);
\draw[fill] (9.00,11.79) circle (0.15);
\draw[fill] (15.00,6.63) circle (0.15);
\draw[fill] (15.00,1.47) circle (0.15);
\draw[fill] (15.00,8.11) circle (0.15);
\draw[fill] (15.00,5.16) circle (0.15);
\draw[fill] (21.00,9.33) circle (0.15);
\draw[fill] (9.00,1.47) circle (0.15);
\draw[fill] (9.00,2.21) circle (0.15);
\draw[fill] (3.00,4.67) circle (0.15);
\draw[fill] (15.00,2.21) circle (0.15);
\draw[fill] (15.00,2.95) circle (0.15);
\draw[fill] (9.00,12.53) circle (0.15);
\draw[fill] (3.00,9.33) circle (0.15);
\draw[fill] (9.00,8.11) circle (0.15);
\draw[fill] (9.00,8.84) circle (0.15);
\draw[fill] (9.00,10.32) circle (0.15);
\draw[fill] (9.00,5.89) circle (0.15);
\draw[fill] (15.00,4.42) circle (0.15);
\draw[fill] (9.00,11.05) circle (0.15);
\draw[fill] (15.00,9.58) circle (0.15);
\draw[fill] (9.00,3.68) circle (0.15);
\draw[fill] (15.00,10.32) circle (0.15);
\draw[fill] (15.00,7.37) circle (0.15);
\draw[fill] (9.00,4.42) circle (0.15);
\draw[fill] (9.00,6.63) circle (0.15);
\draw[fill] (15.00,11.79) circle (0.15);
\draw[fill] (9.00,13.26) circle (0.15);
\draw[fill] (15.00,5.89) circle (0.15);
\draw[fill] (21.00,4.67) circle (0.15);
\draw[fill] (15.00,8.84) circle (0.15);
\draw[fill] (0.00,7.00) circle (0.15);
\draw[fill] (24.00,7.00) circle (0.15);
\draw (18.00,4.00)--(15.00,1.47);
\draw (18.00,4.00)--(15.00,8.11);
\draw (18.00,4.00)--(15.00,5.16);
\draw (18.00,4.00)--(21.00,9.33);
\draw (12.00,0.61)--(9.00,0.74);
\draw (12.00,0.61)--(9.00,7.37);
\draw (12.00,0.61)--(15.00,0.74);
\draw (12.00,0.61)--(15.00,1.47);
\draw (6.00,2.00)--(9.00,0.74);
\draw (6.00,2.00)--(9.00,1.47);
\draw (6.00,2.00)--(9.00,2.21);
\draw (6.00,2.00)--(3.00,4.67);
\draw (12.00,1.22)--(9.00,0.74);
\draw (12.00,1.22)--(15.00,2.21);
\draw (12.00,1.22)--(15.00,2.95);
\draw (12.00,1.22)--(9.00,12.53);
\draw (6.00,8.00)--(9.00,7.37);
\draw (6.00,8.00)--(3.00,9.33);
\draw (6.00,8.00)--(9.00,8.11);
\draw (6.00,8.00)--(9.00,8.84);
\draw (12.00,10.96)--(15.00,0.74);
\draw (12.00,10.96)--(9.00,10.32);
\draw (12.00,10.96)--(9.00,5.89);
\draw (12.00,10.96)--(15.00,4.42);
\draw (12.00,7.91)--(9.00,7.37);
\draw (12.00,7.91)--(15.00,10.32);
\draw (12.00,7.91)--(15.00,7.37);
\draw (12.00,7.91)--(9.00,4.42);
\draw (12.00,12.17)--(15.00,1.47);
\draw (12.00,12.17)--(9.00,6.63);
\draw (12.00,12.17)--(15.00,11.79);
\draw (12.00,12.17)--(9.00,13.26);
\draw (18.00,2.00)--(15.00,0.74);
\draw (18.00,2.00)--(15.00,5.89);
\draw (18.00,2.00)--(21.00,4.67);
\draw (18.00,2.00)--(15.00,8.84);
\draw (12.00,5.48)--(9.00,9.58);
\draw (12.00,5.48)--(15.00,3.68);
\draw (12.00,5.48)--(9.00,2.95);
\draw (12.00,5.48)--(15.00,8.11);
\draw (6.00,10.00)--(9.00,9.58);
\draw (6.00,10.00)--(3.00,9.33);
\draw (6.00,10.00)--(9.00,10.32);
\draw (6.00,10.00)--(9.00,11.05);
\draw (12.00,12.78)--(9.00,9.58);
\draw (12.00,12.78)--(15.00,10.32);
\draw (12.00,12.78)--(9.00,6.63);
\draw (12.00,12.78)--(15.00,5.89);
\draw (18.00,6.00)--(15.00,3.68);
\draw (18.00,6.00)--(15.00,2.21);
\draw (18.00,6.00)--(15.00,9.58);
\draw (18.00,6.00)--(21.00,4.67);
\draw (12.00,6.09)--(9.00,2.95);
\draw (12.00,6.09)--(15.00,2.95);
\draw (12.00,6.09)--(9.00,8.11);
\draw (12.00,6.09)--(15.00,11.79);
\draw (12.00,1.83)--(15.00,3.68);
\draw (12.00,1.83)--(9.00,1.47);
\draw (12.00,1.83)--(15.00,4.42);
\draw (12.00,1.83)--(9.00,13.26);
\draw (12.00,3.65)--(15.00,8.11);
\draw (12.00,3.65)--(9.00,2.21);
\draw (12.00,3.65)--(9.00,8.84);
\draw (12.00,3.65)--(15.00,8.84);
\draw (6.00,4.00)--(9.00,2.95);
\draw (6.00,4.00)--(3.00,4.67);
\draw (6.00,4.00)--(9.00,3.68);
\draw (6.00,4.00)--(9.00,4.42);
\draw (12.00,9.74)--(15.00,12.53);
\draw (12.00,9.74)--(9.00,5.16);
\draw (12.00,9.74)--(9.00,11.79);
\draw (12.00,9.74)--(15.00,5.16);
\draw (12.00,6.70)--(15.00,12.53);
\draw (12.00,6.70)--(9.00,8.84);
\draw (12.00,6.70)--(15.00,4.42);
\draw (12.00,6.70)--(9.00,3.68);
\draw (18.00,12.00)--(15.00,12.53);
\draw (18.00,12.00)--(15.00,7.37);
\draw (18.00,12.00)--(15.00,11.79);
\draw (18.00,12.00)--(21.00,4.67);
\draw (12.00,10.35)--(9.00,5.16);
\draw (12.00,10.35)--(15.00,2.95);
\draw (12.00,10.35)--(9.00,11.05);
\draw (12.00,10.35)--(15.00,8.84);
\draw (6.00,12.00)--(9.00,11.79);
\draw (6.00,12.00)--(9.00,12.53);
\draw (6.00,12.00)--(3.00,9.33);
\draw (6.00,12.00)--(9.00,13.26);
\draw (12.00,8.52)--(15.00,5.16);
\draw (12.00,8.52)--(15.00,2.21);
\draw (12.00,8.52)--(9.00,10.32);
\draw (12.00,8.52)--(9.00,4.42);
\draw (6.00,6.00)--(9.00,5.16);
\draw (6.00,6.00)--(3.00,4.67);
\draw (6.00,6.00)--(9.00,5.89);
\draw (6.00,6.00)--(9.00,6.63);
\draw (12.00,2.43)--(15.00,5.16);
\draw (12.00,2.43)--(9.00,1.47);
\draw (12.00,2.43)--(9.00,8.11);
\draw (12.00,2.43)--(15.00,5.89);
\draw (12.00,4.26)--(9.00,11.79);
\draw (12.00,4.26)--(9.00,2.21);
\draw (12.00,4.26)--(15.00,9.58);
\draw (12.00,4.26)--(15.00,10.32);
\draw (18.00,10.00)--(15.00,13.26);
\draw (18.00,10.00)--(15.00,11.05);
\draw (18.00,10.00)--(15.00,6.63);
\draw (18.00,10.00)--(21.00,9.33);
\draw (12.00,11.57)--(15.00,13.26);
\draw (12.00,11.57)--(9.00,8.11);
\draw (12.00,11.57)--(9.00,5.89);
\draw (12.00,11.57)--(15.00,9.58);
\draw (12.00,9.13)--(15.00,13.26);
\draw (12.00,9.13)--(9.00,4.42);
\draw (12.00,9.13)--(9.00,13.26);
\draw (12.00,9.13)--(15.00,8.84);
\draw (12.00,7.30)--(15.00,11.05);
\draw (12.00,7.30)--(9.00,12.53);
\draw (12.00,7.30)--(9.00,3.68);
\draw (12.00,7.30)--(15.00,5.89);
\draw (12.00,13.39)--(15.00,6.63);
\draw (12.00,13.39)--(15.00,2.21);
\draw (12.00,13.39)--(9.00,8.84);
\draw (12.00,13.39)--(9.00,6.63);
\draw (18.00,8.00)--(21.00,9.33);
\draw (18.00,8.00)--(15.00,2.95);
\draw (18.00,8.00)--(15.00,4.42);
\draw (18.00,8.00)--(15.00,10.32);
\draw (12.00,4.87)--(15.00,11.05);
\draw (12.00,4.87)--(9.00,2.21);
\draw (12.00,4.87)--(9.00,10.32);
\draw (12.00,4.87)--(15.00,11.79);
\draw (12.00,3.04)--(15.00,6.63);
\draw (12.00,3.04)--(9.00,1.47);
\draw (12.00,3.04)--(9.00,11.05);
\draw (12.00,3.04)--(15.00,7.37);
\draw[thick] (9.00,0.74) .. controls (9.90,5.16) and (9.90,5.16) .. (9.00,9.58);
\draw[thick] (15.00,12.53) .. controls (15.90,12.89) and (15.90,12.89) .. (15.00,13.26);
\draw[thick] (9.00,7.37) .. controls (9.90,6.26) and (9.90,6.26) .. (9.00,5.16);
\draw[thick] (15.00,3.68) .. controls (15.90,7.37) and (15.90,7.37) .. (15.00,11.05);
\draw[thick] (15.00,0.74) .. controls (15.90,3.68) and (15.90,3.68) .. (15.00,6.63);
\draw[thick] (9.00,2.95) .. controls (9.90,7.37) and (9.90,7.37) .. (9.00,11.79);
\draw[thick] (15.00,1.47) .. controls (15.90,5.53) and (15.90,5.53) .. (15.00,9.58);
\draw[thick] (15.00,8.11) .. controls (15.90,7.74) and (15.90,7.74) .. (15.00,7.37);
\draw (21.00,9.33)--(24.00,7.00);
\draw (3.00,4.67)--(0.00,7.00);
\draw[thick] (9.00,12.53) .. controls (9.90,9.21) and (9.90,9.21) .. (9.00,5.89);
\draw (3.00,9.33)--(0.00,7.00);
\draw[thick] (9.00,11.05) .. controls (9.90,7.37) and (9.90,7.37) .. (9.00,3.68);
\draw (21.00,4.67)--(24.00,7.00);

\end{tikzpicture}  
\caption{The example showing that $f(4,7)<\frac{n(4,7)}{7}=\frac{67}{7}.$}\label{fig:omega3delta4}
\label{fig:almost4RegGirth7}
\end{figure}
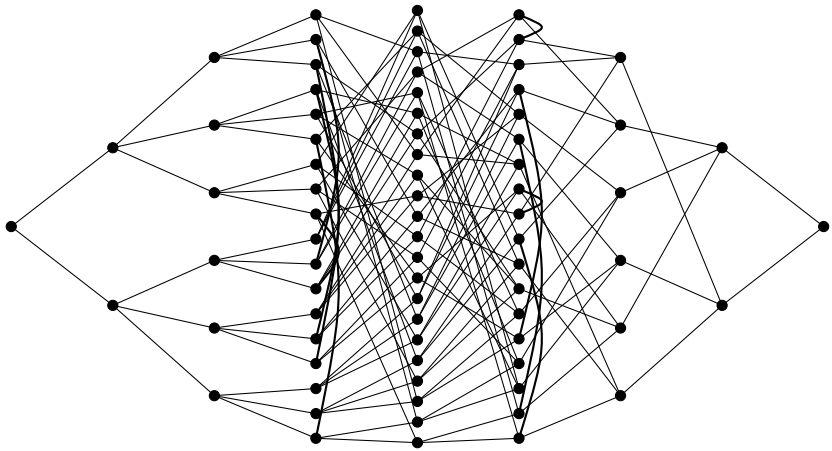
\end{proof}

Together, \cref{prop:imprLowerBound,prop:imprLowerBound_specialimprovement,prop:upp_bnd,prop:main_not_sharp} yield \cref{thm:general_estimate}.

\section{Proof of Theorem~\ref{thm:finiteTask}}\label{sec:thmFiniteTask}

We now define the notion of a \textit{repeatable graph} (this notion is based on a similar notion as discussed in~\citep{CJ25}, but slightly adjusted to fit the precise context of the current paper). We will define this notion in such a way that every repeatable graph (with respect to $k$ and $g$) gives rise to an upper bound on $f(k,g)$.

\begin{definition}
Let $G$ be a graph and let $N_0 \subset V(G)$ be a subset of vertices. For all $i \geq 1$, define $N_i = N_i(N_0) := \{v \in V(G)~|~\min_{u \in N_0}(d(u,v))=i\}$. Let $d$ be the largest integer for which $N_d$ is non-empty. For integers $k$ and $g$, we call a graph $G$ induced by consecutive neighbourhoods $N_0, N_1, \ldots,N_d$ a repeatable graph (with respect to $k$ and $g$) if 
\begin{enumerate}[noitemsep]
	\item $G$ has girth at least $g$, 
	\item every vertex in $N_1 \cup \ldots \cup N_{d-1}$ has degree $k$, 
	\item the sum of the degrees of vertices in $N_0$ is even whenever $k$ is even,
	\item the sum of the degrees of vertices in $N_d$ is even whenever $k$ is even, 
	\item $d+1 \geq 2g$ and
	\item  there exists an isomorphism between the graph induced by $N_0 \cup \ldots \cup N_{g-1}$ and the graph induced by $N_{d-g+1} \cup \ldots \cup N_d$ such that for every $\ell \in \{0,1,\ldots,g-1\}$ the isomorphism maps every vertex in $N_\ell$ to a vertex in $N_{\ell+d-g+1}$. 
\end{enumerate}

\end{definition}

The graph in Fig.~\ref{fig:repeatableK3G4} shows an example of a graph that is repeatable with respect to $k=3$ and $g=4$.

\begin{figure}[h]
    \centering

    \begin{tikzpicture}[scale=0.6]

    \foreach \x in {0}{
        \foreach \y in {-1,1}{
        \draw[fill] (\x,\y) circle (0.25);
        }
    }
    \foreach \x in {2}{
        \foreach \y in {0}{
        \draw[fill] (\x,\y) circle (0.15);
        }
    }
    \foreach \x in {4}{
        \foreach \y in {0}{
        \draw[fill] (\x,\y) circle (0.15);
        }
    }
    \foreach \x in {6}{
        \foreach \y in {-1,1}{
        \draw[fill] (\x,\y) circle (0.15);
        }
    }
    \foreach \x in {8}{
        \foreach \y in {-1,1}{
        \draw[fill] (\x,\y) circle (0.15);
        }
    }
    \foreach \x in {10}{
        \foreach \y in {0}{
        \draw[fill] (\x,\y) circle (0.15);
        }
    }
    \foreach \x in {12}{
        \foreach \y in {0}{
        \draw[fill] (\x,\y) circle (0.15);
        }
    }
    \foreach \x in {14}{
        \foreach \y in {-1,1}{
        \draw[fill] (\x,\y) circle (0.15);
        }
    }

   \draw (2,0)--(0,-1);
   \draw (2,0)--(0,1);
   \draw (4,0)--(2,0);
   \draw (6,-1)--(4,0);
   \draw (6,1)--(4,0);
   \draw (8,-1)--(6,-1);
   \draw (8,-1)--(6,1);
   \draw (8,1)--(6,1);
   \draw (8,1)--(6,-1);
   \draw (10,0)--(8,-1);
   \draw (10,0)--(8,1);
   \draw (12,0)--(10,0);
   \draw (14,-1)--(12,0);
   \draw (14,1)--(12,0);
    \end{tikzpicture} 
\caption{A graph that is repeatable with respect to $k=3$ and $g=4$ (where vertices in $N_0$ are shown larger).}\label{fig:repeatableK3G4}
\end{figure}
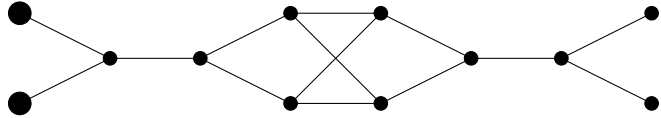

We now show that repeatable graphs indeed give rise to an upper bound on $f(k,g)$.
Here, we will make use of the following two near-trivial lemmas, that can be used to extend partial constructions to desired ones.

\begin{lemma}
\label{lem:extensionLemma1}
    Given any sequence of natural numbers such that the sum is even, there is a (not necessarily  connected) multigraph which has these numbers as its degree sequence.    
\end{lemma}

\begin{proof}
    This is trivial, since once can draw iteratively edges between any two vertices which did not attain the desired degree yet (even if it is the vertex itself and so a loop is created).
\end{proof}

\begin{lemma}\label{lem:extensionLemma2}
    For every $k,g \ge 3$, any $k$-regular multigraph $G$ can be modified into a $(k,g)$-graph $G'$, by replacing some of its edges with a fixed graph.
\end{lemma}

\begin{proof}
    Let $H$ be any $\kg$-graph (one can take a $\kg$-cage), and $H^-$ be $H$ with one edge removed.

    Every edge $ab\in E(G)$ (or iteratively some of its edges which belong to cycles of length smaller than $g$) can be removed, after which $a$ and $b$ are connected to the degree $k-1$ vertices of a copy of $H^-$.

    By construction, the resulting graph is still $k$-regular, and has girth at least $g$, since $H$ has this property (note that every cycle in $G'$ using initial edges has the property by definition).  

    By choosing $H^-$ such that it has a cycle of length exactly $g$ (and executing one replacement if the girth of $G$ is larger than $g$), we can ensure that $G$ has girth exactly $g.$
\end{proof}

\begin{proposition}
\label{prop:repeatableGraphYieldsUpperBound}
For integers $k$ and $g$, let $G$ be induced by consecutive neighbourhoods $N_0, N_1, \ldots, N_d$ be a repeatable graph with respect to $k$ and $g$. Then $f(k,g) \leq \frac{|N_0 \cup \ldots \cup N_{d-g}|}{d+1-g}$.
\end{proposition}
\begin{proof}
    Let $G'$ be the graph obtained by taking two disjoint copies $G_1$ and $G_2$ of $G$ and identifying the vertices in neighbourhood $N_{\ell+d-g+1}$ of $G_2$ with the vertices in neighbourhood $N_\ell$ of $G_1$ for every $\ell \in \{0,1,\ldots,g-1\}$ (according to the isomorphism that maps the vertices in these neighbourhoods to each other). 
    Now notice that by this construction $G'$ is again repeatable with respect to $k$ and $g$. By iteratively applying this construction, we obtain (repeatable) graphs for which the ratio between the order and the diameter goes to $\frac{|N_0 \cup \ldots \cup N_{d-g}|}{d+1-g}$. Moreover, we can complete any of these graphs to $(k,g)$-graphs by only introducing $O_{k,g}(1)$ extra vertices using the extension lemmas~\ref{lem:extensionLemma1},~\ref{lem:extensionLemma2}.

    We illustrate this for the extension at $N_0$.
    By~\cref{lem:extensionLemma1} applied to the degree sequence $\{ k-\deg(v) \mid v \in N_0\}$ (possibly with $k$ appended to this sequence if $k$ is odd to correct the parity), we can add edges between the vertices in $N_0$ (and one optional additional vertex if necessary) such that it becomes $k$-regular.
    Next, we can apply~\cref{lem:extensionLemma2} on the graph induced by the new edges to satisfy the girth condition.
    Finally, one can also apply this extension to $N_d$ to obtain a $(k,g)$-graph.
\end{proof}

We now prove~\cref{thm:finiteTask} based on the notion of repeatable graphs.

%\begin{proof}[Proof of~\cref{thm:finiteTask}]
\begin{proof}[of~\cref{thm:finiteTask}]
We will show that there exists a graph induced by consecutive neighbourhoods $N_0, N_1, \ldots, N_d$ that is repeatable with respect to $k$ and $g$ for which $f(k,g) = \frac{|N_0 \cup \ldots \cup N_{d-g}|}{d+1-g}$ and that it suffices to only consider $O_{k,g}(1)$ different graphs to find such a graph. Hence, $f(k,g)$ can be computed in $O_{k,g}(1)$ time.

Consider a $(k;g,d)$-cage for $d$ sufficiently large. Since $f(k,g) \le \frac{n(k,g)}{g},$ at least $1/2-o(1)$ of all tuples of $g$ consecutive neighbourhoods have a sum of sizes bounded by $2n(k,g).$ 
In the latter case, there are $O_{k,g}(1)$ possibilities for the graph induced by these $g$ consecutive neighbourhoods. By the pigeonhole principle, this means that if we consider sufficiently many (but bounded number of) consecutive neighbourhoods, we can find some, $N_a, N_{a+1}, \ldots, N_b$, such that the graph induced by $N_a \cup N_{a+1} \cup \ldots \cup N_b$ must be a repeatable graph. So both the diameter of this graph as well as the neighbourhood sizes of this graph are all bounded by $O_{k,g}(1)$, which means there are only $O_{k,g}(1)$ such graphs in total. We will now argue that at least one of those $O_{k,g}(1)$ graphs achieves the equality.

By choosing $d$ arbitrarily large, we can obtain arbitrarily many repeatable subgraphs. We now claim that it is impossible that for all of these graphs the ratio $\frac{|N_a \cup \ldots \cup N_{b-g}|}{b-a+1-g}$ exceeds $f(k,g)$ (clearly the ratio cannot be smaller than $f(k,g)$ because of Proposition~\ref{prop:repeatableGraphYieldsUpperBound}).

We now define the notion of removing a repeatable graph: if $G[ \cup_{i=a}^b N_i]$ is repeatable, removing it implies we replace $G$ by the graph $G'$ which is formed by adding edges to the disjoint union of $G[\cup_{i=0}^{a+g-2} N_i]$ and $G[\cup_{i=b}^{d} N_i]$ between $N_{a+g-2}$ and $N_b$ such that $G'[\cup_{i=a}^{a+g-2} N_i \cup N_b] \cong G[\cup_{i=b-g+1}^{b} N_i]$ and the isomorphism maps the vertices of consecutive neighbourhoods to each other. Note that this yields again a $k$-regular graph of girth at least $g$. Fig.~\ref{fig:removingRepeatableSubgraph} shows an example where the repeatable graph from Fig.~\ref{fig:repeatableK3G4} is removed from a larger graph $G$.

\begin{figure}[h]
    \centering

    \begin{tikzpicture}[scale=0.66]

    \foreach \x in {-1}{
        \foreach \y in {-1,-0.333,0.333,1}{
        \draw[fill] (\x,\y) circle (0.15);
        }
    }
    
    \foreach \x in {0}{
        \foreach \y in {-1,1}{
        \draw[fill] (\x,\y) circle (0.15);
        }
    }
    \foreach \x in {1}{
        \foreach \y in {0}{
        \draw[fill] (\x,\y) circle (0.15);
        }
    }
    \foreach \x in {2}{
        \foreach \y in {0}{
        \draw[fill] (\x,\y) circle (0.15);
        }
    }
    \foreach \x in {3}{
        \foreach \y in {-1,1}{
        \draw[fill] (\x,\y) circle (0.15);
        }
    }
    \foreach \x in {4}{
        \foreach \y in {-1,1}{
        \draw[fill] (\x,\y) circle (0.15);
        }
    }
    \foreach \x in {5}{
        \foreach \y in {0}{
        \draw[fill] (\x,\y) circle (0.15);
        }
    }
    \foreach \x in {6}{
        \foreach \y in {0}{
        \draw[fill] (\x,\y) circle (0.15);
        }
    }
    \foreach \x in {7}{
        \foreach \y in {-1,1}{
        \draw[fill] (\x,\y) circle (0.15);
        }
    }

    \foreach \x in {8}{
        \foreach \y in {-1,-0.333,0.333,1}{
        \draw[fill] (\x,\y) circle (0.15);
        }
    }
    
   \draw (1,0)--(0,-1);
   \draw (1,0)--(0,1);
   \draw (2,0)--(1,0);
   \draw (3,-1)--(2,0);
   \draw (3,1)--(2,0);
   \draw (4,-1)--(3,-1);
   \draw (4,-1)--(3,1);
   \draw (4,1)--(3,1);
   \draw (4,1)--(3,-1);
   \draw (5,0)--(4,-1);
   \draw (5,0)--(4,1);
   \draw (6,0)--(5,0);
   \draw (7,-1)--(6,0);
   \draw (7,1)--(6,0);

    % left
   \draw (-1,-1)--(-1,-0.333);
   \draw (-1,-0.333)--(-1,0.333);
   \draw (-1,-1) .. controls +(-0.5,-0.0) and +(-0.5,-0.0) .. (-1,1); 
   \draw (-1,1)--(-1,0.333);
   
   \draw (0,-1)--(-1,-0.333); 
   \draw (0,-1)--(-1,1); 
   \draw (0,1)--(-1,-1); 
   \draw (0,1)--(-1,0.333); 

   % right
   \draw (8,-1)--(8,-0.333);
   \draw (8,-0.333)--(8,0.333);
   \draw (8,-1) .. controls +(+0.5,-0.0) and +(+0.5,-0.0) .. (8,1); 
    \draw (8,1)--(8,0.333);
    
   \draw (7,-1)--(8,-0.333); 
   \draw (7,-1)--(8,1); 
   \draw (7,1)--(8,-1); 
   \draw (7,1)--(8,0.333); 

    % arrow
   \draw[->,thick] (9,0) -- ++(1,0);

   \foreach \x in {11}{
        \foreach \y in {-1,-0.333,0.333,1}{
        \draw[fill] (\x,\y) circle (0.15);
        }
    }
    
    \foreach \x in {12}{
        \foreach \y in {-1,1}{
        \draw[fill] (\x,\y) circle (0.15);
        }
    }
    \foreach \x in {13}{
        \foreach \y in {0}{
        \draw[fill] (\x,\y) circle (0.15);
        }
    }
    \foreach \x in {14}{
        \foreach \y in {0}{
        \draw[fill] (\x,\y) circle (0.15);
        }
    }
    \foreach \x in {15}{
        \foreach \y in {-1,1}{
        \draw[fill] (\x,\y) circle (0.15);
        }
    }

    \foreach \x in {16}{
        \foreach \y in {-1,-0.333,0.333,1}{
        \draw[fill] (\x,\y) circle (0.15);
        }
    }
    
   \draw (13,0)--(12,-1);
   \draw (13,0)--(12,1);
   \draw (14,0)--(13,0);
   \draw (15,-1)--(14,0);
   \draw (15,1)--(14,0);

    % left
   \draw (11,-1)--(11,-0.333);
   \draw (11,-0.333)--(11,0.333);
   \draw (11,-1) .. controls +(-0.5,-0.0) and +(-0.5,-0.0) .. (11,1); 
   \draw (11,1)--(11,0.333);
   
   \draw (12,-1)--(11,-0.333); 
   \draw (12,-1)--(11,1); 
   \draw (12,1)--(11,-1); 
   \draw (12,1)--(11,0.333); 

   % right
   \draw (16,-1)--(16,-0.333);
   \draw (16,-0.333)--(16,0.333);
   \draw (16,-1) .. controls +(+0.5,-0.0) and +(+0.5,-0.0) .. (16,1); 
    \draw (16,1)--(16,0.333);
    
   \draw (15,-1)--(16,-0.333); 
   \draw (15,-1)--(16,1); 
   \draw (15,1)--(16,-1); 
   \draw (15,1)--(16,0.333); 
   
    \end{tikzpicture} 
\caption{Removing a repeatable subgraph ($k=3$, $g=4$, $a=1$, $b=8$, $d=9$).}\label{fig:removingRepeatableSubgraph}
\end{figure}
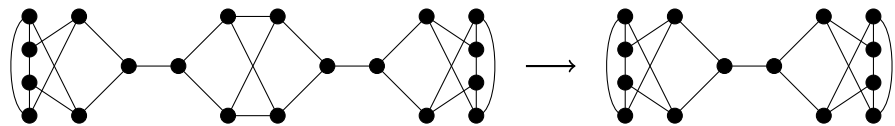

    Removing iteratively a repeatable subgraph of minimum length, if all ratios $\frac{|N_a \cup \ldots \cup N_{b-g}|}{b-a+1-g}$ would exceed $f(k,g)$, one deduces a contradiction as follows.
    (This is similar to~\cite[page 3]{CJ25}, with the difference that not every neighbourhood needs to be bounded.) Due to the finite number of combinations, the ratio $\frac{|N_a \cup \ldots \cup N_{b-g}|}{b-a+1-g}$ is always at least a fixed constant $c> f(k,g).$
    One can repeat the removal process until the diameter of the resulting graph is smaller than some fixed value $x$, depending on the $O_{k,g}(1)$ term.
    
    Initially, we start with a graph of order $n=f(k,g)d+O_{k,g}(1).$
    After removing sufficiently many repeatable graphs, the diameter is $d'<x$ and the order is $n'<f(k,g)d-c(d-d')+O_{k,g}(1)\le  f(k,g)x - (c-f(k,g))(d-x)+O_{k,g}(1) .$ 
    % Here $d'$ can be bounded in terms of the other parameters.
    Since $d$ could be chosen arbitrarily large at the start, for $d$ sufficiently large we have $n'<0$, which is a contradiction.
\end{proof}

\section{Determining $n(k;g,d)$ for several open cases}\label{sec:determined_bounds}

In this section, we first theoretically determine the values of $n(k; g,d)$, and count the number of girth-diameter cages for $(k,g)=(3,4)$ and $(k,g)=(3,5)$ in \cref{sec:k3g4}. Then, in \cref{sec:gen_algo}, we present an exhaustive generation algorithm for $\kgd$-graphs and use it to computationally determine several more values of $n\kgd$ and the corresponding $\kgd$-cages.

%In this section, we will determine the values for $n(k; g,d)$, and count the number of girth-diameter cages when $(k,g) \in \{(3,4),(3,5)\}.$ 

\subsection{Theoretical determination of $n(3;g,d)$ for $g \in \{4,5\}$}\label{sec:k3g4}

In what follows, we determine $n\kgd$ when $(k,g) \in \{(3,4),(3,5)\}$.
Hence the $O_{k,g}(1)$ term for these values is determined exactly. 
The fact that this is in general a hard task, can be concluded from the computational results in Appendix~\ref{sec:tables_generator}, e.g.\ observe the values for $(k,g)=(3,7)$ when $7 \mid d$;
for $d \in \{7,14,21\},$ the lower bound $M(k; g,d)$ is sharp, but for $d \in \{28,35\}$ it is not. 
The data seems to suggest that $n(k; g,d)=n(k; g,d-7)+24$ for $d \ge 23$ and all further results may be analogous to the cases where $(k,g) \in \{(3,4),(3,5)\}.$

\begin{proposition}\label{prop:kg=34scherp}
    Let $d = 5+4k+j$, where $0 \le j \le 3$ and $k\ge 0.$ 
    A $(3;4,d)$-cage $G$ has at least $k$ bridges and the order of $G$ is $n=14+2j+6k.$ 
\end{proposition}

\begin{proof}
    Let $u$ and $v$ be vertices such that $d(u,v)=d.$
    Notice that $\abs{N_1(u)}=\abs{N_1(v)}=3$ and $\abs{N_2(u)}, \abs{N_2(v)} \ge 3$ (if $N_1(u)$ has exactly $3$ neighbours, $G$ being regular would imply that $G=K_{3,3}$ and $d=2$).
    If $G$ has no bridges (equivalently, no cutvertices), then $\abs{N_i(u)} \ge 2$ for every $3\le i \le d-3$, which implies $n\ge 14+2j+8k.$
    If $G$ has bridges, we consider the components of such a graph with the bridges removed. Here we distinguish the two end blocks (i.e., first and last neighbourhood) containing one of the two vertices $u$ and $v$, and the gadgets which contain two vertices of degree two. A gadget which has distance $i-1$ between the degree two vertices and thus contributing $i$ to the diameter, has either order at least $8>2i$ for $i\in \{2,3\}$, or order at least $2i-2$ if $i \ge 4$.
    An end block containing $u$, where $u$ has eccentricity $i$, contains at least $2i+3$ vertices.
    The number of gadgets for which the order is less than twice the contributed distance, is upper bounded by $k$.
    This results in $n \ge 2d+4-2k,$ as desired.
    By inserting gadgets $K_{3,3}^-$ ($K_{3,3}$ minus one edge) to $(3;4,d)$-cages for small $d$, it is easy to see that the bound is tight.
\end{proof}

With a careful analysis of components, we can also characterise the $(3;4,d)$-cages, and in particular count them.
%The proof of the following result is sketched in \cref{sec:proof_kg34scherp} %~\cref{sec:fullproof} (mostly counting the possible number of combinations of components of the cage after bridges are removed).

\begin{proposition}\label{prop:kg34scherp}
    The number of $(3;4,d)$-cages for $d\ge 9$ equals
    $$
    \begin{cases}
	1  & \text{for } d \equiv 1 \pmod 4 \\
        4 & \text{for } d \equiv 2 \pmod 4\\
        17+\floorfrac{d}{8} & \text{for } d \equiv 3 \pmod 4\\
        27+d+\floorfrac{d-4}{8} & \text{for } d \equiv 0 \pmod{4}.
	\end{cases}
	$$
\end{proposition}

\begin{proof}%[Proof of~\cref{prop:kg34scherp}]
    Since we study the case $d \ge 9$, by~\cref{prop:kg=34scherp}, we know that every $(3;4,d)$-cage has at least one bridge.
    For the end blocks, we can consider the sizes $(1, \abs{N_1(u)}, \abs{N_2(u)}, \ldots, \abs{N_i(u)})$ (and analogous for $v$).
    Since the end block needs to have $2i+3$ vertices in total and have $i \le 6$, taking into account the other conditions, there are only a few possibilities for the tuples. The following pairs consist of 1) a string representing a concatenation of the neighbourhood sizes and 2) the number of possibilities for the end block:
$$(133, 1), 
(1341, 3),
(1332, 1), 
(13421, 5),$$
$$(13331, 2), 
(134221, 3), 
(133231, 2), 
(133321, 2) .$$

The gadgets are unique based on the sizes 
\begin{equation*}
    (1,2,2,1),(1,2,2,2,2,1) \text{ or } (1,2,2,2,2,2,1)
\end{equation*}
of the neighbourhoods $N_i(w)$, where $w$ is a degree two vertex. 

For $d\equiv 1 \pmod 4,$ the only optimal combination is when two $133$s are combined with $1221$s, which gives a unique cage.

\begin{figure}[h]
    \centering

    \begin{tikzpicture}[scale=1]

        \foreach \y in {-1,0,1}{
        \draw[fill] (1,\y) circle (0.15);
        \draw[fill] (2,\y) circle (0.15);
        \draw (0,0)--(1,\y);
        }

        \draw (1,0)--(2,1)--(1,1);
        \draw (2,1)--(1,-1);
        \draw (1,1)--(2,0)--(2,-1);
        \draw(1,-1)--(2,-1)--(1,0);

        \foreach \y in {-0.5,0.5}{
        \draw[fill] (4,\y) circle (0.15);
        \draw[fill] (5,\y) circle (0.15);
        \draw (6,0)--(5,0.5)--(4,\y)--(3,0);
        \draw (6,0)--(5,-0.5)--(4,\y);
        }
\draw[fill] (0,0) circle (0.15);
        \draw[fill] (3,0) circle (0.15);
        \draw[fill] (6,0) circle (0.15);
   
    \end{tikzpicture} 
\caption{Example of end block and gadget which are decoded by $133$ and $1221$.}\label{fig:133&1221}
\end{figure}
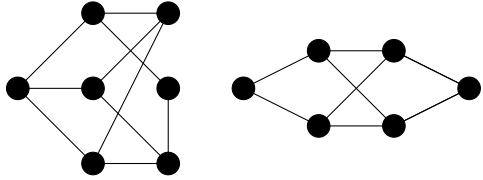

For $d \equiv 2 \pmod 4$, we can combine one $1341$ or $1332$ with a $133$ and some $1221$s.\\
For $d \equiv 3 \pmod 4$, not mentioning the $1221$s,
one can have the combinations
$133+122221+133$, $1341+1341$, $1341+1332$, $1332+1332$, $13421+133$ and $13331+133$.
This results in $\floorfrac{d}{8}+6+3+1+5+2=17+\floorfrac{d}{8}$ possible girth-diameter cages.

For $d\equiv 0 \pmod 4,$ the options are
$133+1222221+133$, $133+122221+1332$, $133+122221+1341$, 
$1341+13421$, $1332+13421$, $1341+13331$, $1332+13331$, $134221+133$, $133231+133$ and $133321+133$.
This results in $\floorfrac{d-4}{8}+(3+1)\frac{d-4}{8}+(1+3)(5+2)+(3+2+2)=27+d+\floorfrac{d-4}{8}$ possible girth-diameter cages.
\end{proof}

\begin{proposition}\label{prop:kg=35scherp}
    For $d \ge 5$, $n(3;5,d) = 2d+10$ if $d \equiv 0,1 \pmod 5$ and $n(3;5,d)= 2d+8$ if $d \equiv 2,3,4 \pmod 5$.
    %For $d \ge 5$, a $(3;5,d)$-graph $G$ satisfies $n \ge 2d+10$ if $d \equiv 0,1 \pmod 5$ and $n \ge 2d+8$ if $d \equiv 2,3,4 \pmod 5$.
    For $2 \le d \le 4$, $n(3;5,d)= 2d+6.$
\end{proposition}

\begin{proof}
    The proof of~\cref{prop:imprLowerBound} explained $\sum_{i=1}^{j} |N_{q+i}| \geq M(k,j)$, which for $j=5$ and $k=3$ implies here that $5$ consecutive neighbourhoods cover at least $M(3,5)=10$ vertices. We also note that $\abs{N_0}+\abs{N_1}\ge 4$, $\abs{N_0}+\abs{N_1}+\abs{N_2} \ge 10$, $\abs{N_d}+\abs{N_{d-1}}+\abs{N_{d-2}} \ge 10$, and
    $\abs{N_i} \ge 1$, $\abs{N_i}+\abs{N_{i+1}}+\abs{N_{i+2}}\ge 4$ for every $0\le i\le d-2$. 
    Finally, we know that the order of a cubic graph is even. 
    This implies the lower bound for all cases.

    For $ 2\le d \le 4$, equality is attained by e.g.\ the Petersen graph, Twinplex graph and Banffrican graph (graphs with HoG-id $660, 27409, 50487$ at the House of Graphs~\citep{HOG}).
    
    For $d \ge 5,$
    we can paste Petersen graphs with a little modification together, to match the observed bounds. 
    An example is presented in~\cref{fig:n34d13} for $d \equiv 3 \pmod 5$ (by modifying the left or right to make it more symmetric, one also obtains $d \equiv 2,4 \pmod 5$), where it is easy to adjust the center with one less or adding multiple copies of $P_{5,2}^-$ (Petersen graph minus one edge).

    For $d \equiv 0,1 \pmod 5$, we can adjust the construction as presented (in~\cref{fig:n34d13}) with the red curvy lines replacing $5$ other edges. 
\end{proof}

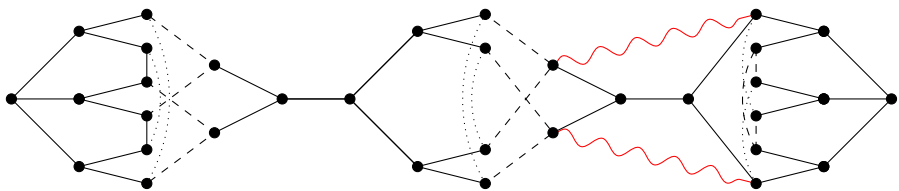
\begin{figure}[h]
    \centering
    \begin{tikzpicture}[scale=0.45]
           \foreach \x in {0,2,8,10,18,20,24,26}{
        \draw[fill] (\x,0) circle (0.15);
    
    }

    \draw[red,decorate, decoration=snake, segment length=5mm] (16,1)--(22,2.5);
    \draw[red,decorate, decoration=snake, segment length=5mm] (16,-1)--(22,-2.5);

    \draw[dotted] (4,1.5) arc (30:-30:3) ;
    \draw[dotted] (4,2.5) arc (30:-30:5) ;

    \draw[dashed] (22,-1.5) arc (210:150:3) ;
    \draw[dotted] (22,-2.5) arc (210:150:3) ;
    \draw[dotted] (22,-0.5) arc (210:150:3) ;

    \draw[dotted] (14,-2.5) arc (210:150:5) ;
    \draw[dotted] (14,-1.5) arc (210:150:3) ;

    \draw[dashed] (22,-1.5)--(22,-0.5);
    \draw[dashed] (22,1.5)--(22,0.5);

    \draw[dashed] (4,-0.5)--(6,1)--(4,2.5);
     \draw[dashed] (4,0.5)--(6,-1)--(4,-2.5);
     \draw (22,-2.5)--(20,0)--(22,2.5);
     \draw (18,0)--(20,0);

    \foreach \y in {0,2,-2}{
        \draw[fill] (2,\y) circle (0.15);
        \draw (0,0)--(2,\y);  
         \draw (26,0)--(24,\y);  
         \draw[fill] (24,\y) circle (0.15);
    }
    
    \foreach \y in {0,2,-2}{
    \draw (4,\y+0.5)--(2,\y);
    \draw (4,\y-0.5)--(2,\y);
        \draw[fill] (4,\y+0.5) circle (0.15);
    \draw[fill] (4,\y-0.5) circle (0.15);
    }

    \draw[fill] (4,1.5)--(4,0.5);
    \draw[fill] (4,-1.5)--(4,-.5);

    \foreach \y in {1,-1}{
        \draw[fill] (16,\y) circle (0.15); 
        \draw[fill] (6,\y) circle (0.15);  
        \draw (6,\y)--(8,0)--(10,0);
    }

    \foreach \y in {2,-2}{
        \draw[fill] (12,\y) circle (0.15);
        \draw[fill] (14,\y+0.5) circle (0.15);
        \draw[fill] (14,\y-0.5) circle (0.15);
        \draw (10,0)--(12,\y)--(14,\y+0.5);   
         \draw (10,0)--(12,\y)--(14,\y-0.5);    
    }
    \draw[dashed] (14,-1.5)--(16,1)--(14,2.5);
     \draw[dashed] (14,1.5)--(16,-1)--(14,-2.5);

     \draw (16,1)--(18,0)--(16,-1);

     \foreach \y in {0,2,-2}{
        \draw[fill] (24,\y) circle (0.15);
        \draw[fill] (22,\y+0.5) circle (0.15);
        \draw[fill] (22,\y-0.5) circle (0.15);
        \draw (22,\y-0.5)--(24,\y)--(22,\y+0.5);     
    }
    
    \end{tikzpicture}
    \caption{$(3,5,13)$-graph of order $34$, easily modified to $(3,5,d)$-graph for $10\le d \le 14$.}
    \label{fig:n34d13}
\end{figure}

By a careful analysis, one can also count the number of $(3;5,d)$-cages.

\begin{proposition}\label{prop:kg35scherp}
    The number of $(3;5,d)$-cages for $d\ge 6$ equals
    $$
    \begin{cases}
        d-\floorfrac{d}{10}  & \text{for } d \equiv 1 \pmod 5 \\
		1  & \text{for } d \equiv 2 \pmod 5 \\
        4 & \text{for } d \equiv 3 \pmod 5 \\
        10 & \text{for } d \equiv 4 \pmod 5 \\
        128+11.3d & \text{for } d \equiv 0 \pmod{10} \\
         138.5+11.3d & \text{for } d \equiv 5 \pmod{10}.          
	\end{cases}
	$$
\end{proposition}

\begin{proof}%[Proof of~\cref{prop:kg35scherp}]
    Except for some small values of $d$, by the proof of~\cref{prop:kg=35scherp} and the observation that no $4$ consecutive neighbourhoods can have size $2$, the extremal graphs have bridges.
    For the components of such a graph with the bridges removed, we can consider the ones (end blocks) with the vertices $u,v$ for which $d(u,v)=d$, and the ones (gadgets) with two vertices of degree $2.$

    For an end block, we start with $N_0$ equal to $u$ or $v$
    and consider the next neighbourhoods.
    Then the sizes $(\abs{N_0(v)}, \abs{N_1(v)},\abs{N_2(v)}, \ldots, \abs{N_i(v)})$ (here $i$ is the eccentricity of $v$ in this end block) have the property that $\abs{N_0(v)}=1, \abs{N_1(v)}=3,\abs{N_2(v)}=6$ and $\abs{N_i(v)}=1$.
    If $i\ge 5,$ $\abs{N_j(v)}=2$ for every $3 \le j \le i-1$ different from $i-2$ since $\abs{N_{i-2}(v)}=4$ is necessary. Also $i \le 8$, by the observation on consecutive neighbourhoods of size $2$.

    For a gadget with degree two vertices $w$ and $x$, $(|N_0(w)|, \allowbreak |N_1(w)|, \allowbreak |N_2(w)|, \allowbreak \ldots, \allowbreak |N_i(w)|)$ (here $i$ is the eccentricity of $w$ in this gadget), we have $\abs{N_0(w)}=1, \abs{N_1(w)}=2,\abs{N_2(w)}=4$ and $\abs{N_i(w)}=1, \abs{N_{i-1}(w)}=2,\abs{N_{i-2}(w)}=4$.
    Here $i \le 8$ again and $i=5$ is impossible.
    
    The only sequences that do not have an additional loss compared to the bound $2d+8$ are     
    $(1,3,6,1)$, $\allowbreak (1,3,6,2,1)$, $\allowbreak (1,2,4,2,1).$
    
    The combination of these leads to $d \equiv 2,3,4 \pmod 5$.

    The components with sizes $(1,3,6,1)$ and $(1,2,4,2,1)$ are unique.
    For the end block with sizes $(1,3,6,2,1)$, there are $4$ possibilities.
    This gives $1, 1 \cdot 4=4$ and $\binom{4}{2}+4=10$ possibilities resp.\ for $d \equiv 2,3,4 \pmod 5$.

    A gadget with sizes $(1,3,6,2,2,2,4,1)$ can only be extended to potential extremal graphs with $d \equiv 2,3 \pmod 5$, which is not extremal.

    For the other sizes, we list the number of possible components.
    
    \begin{claim}
        There is only one possible component with sizes $(1,3,6,2,4,2,1)$, \allowbreak $(1,2,4, \allowbreak 2,2,4,2,1)$ or \allowbreak $(1,2,4,2,2,2,4,2,1)$. For sizes $(1,2,4,4,2,1)$ and $(1,3,6,2,2,4,2,1)$ there are $6$ for each tuple.
        For sizes $(1,3,6,4,2,1)$ there are $60$ options.
    \end{claim}

    Not mentioning the filling with the unique optimal gadget $(1,2,4,2,1)$, we have the following options.

    For $d \equiv 0 \pmod{5} $ with $d \ge 10$,
    there are $4 \times 60$ possibilities that combine $(1,3,6,4,2,1)$ and $(1,3,6,2,1)$.
    There is one option using $(1,3,6,2,4,2,1)$ and $(1,3,6,1).$
    There are $\floorfrac{d-5}{10}$ extremal graphs having components $(1,2,4,2,2,4,2,1)$ and two $(1,3,6,1)s.$
    There are $112$ using two $(1,3,6,2,1)$s and a $(1,2,4,2,2,4,2,1)$ at fixed non-central location ($\floorfrac{d-5}{10}$ choices), and $66$ options if the latter is central (which only can happen if $d \equiv 5 \pmod{10}.$

    For $d \equiv 1 \pmod 5$, we have again $4$ compositions.

There are four using $(1,3,6,2,4,2,1)$ and $(1,3,6,2,1),$ and six using $(1,3,6,2, \allowbreak 2,4,2,1)$ and $(1,3,6,1)$.

There are four per location of the $(1,2,4,2,2,4,2,1)$ gadget between end blocks $(1,3,6,1)$ and $(1,3,6,2,1).$
This implies an additional four per increase of $d$ with $5$.
There is one per location of $(1,2,4,2,2,2,4,2,1)$ gadget between two $(1,3,6,1)s.$
By symmetry, the increase happens when $d$ increases by $10.$
\end{proof}

\subsection{Further values of $n(k;g,d)$ with exhaustive generation}\label{sec:gen_algo}

Using the exhaustive generation algorithm for $\kgd$-graphs which we will describe in this section, 
we could determine the value of $n\kgd$ for $\numCombsKgdGenerated$ different triples $(k,g,d)$ of which $\numOrderGenDeterminedNew$ triples are ``new'' (in the sense that they have not yet been determined in \cref{sec:k3g4} of this paper or in the literature). We considered triples with $k,g,d$ in the following sets: $k \in \{3, \ldots, 7\}$, $g \in \{4, \ldots, 8, 10\}$ and $d \in \{2, \ldots, 40\}$. For all $\numCombsKgdGenerated$ triples we determine at least one $\kgd$-cage and we determine the exhaustive set for $\numExhGen$ triples. The $(k;g,d)$-cages we obtained with the algorithm can be downloaded from~\citep{githubRepo} and a selection of these cages can be found on the House of Graphs~\citep{HOG} by searching for the keyword ``girth-diameter cages''.

\newcommand{\backtrack}{recursivelyAddEdges}
\newcommand{\validEdges}{E_{\textrm{add}}}

In the rest of this subsection we will describe our generation algorithm and explain the sanity checks that we performed to verify the correctness of our implementation.

Our algorithm, whose pseudocode can be found in Algorithm~\ref{alg:genGirthDiamGraphs} and~\ref{alg:backtrack}, generates all $(k; g, d)$-graphs of order $n$.
It is based on the generation algorithm for $(k, g)$-graphs by McKay, Myrvold and Nadon~\citep{McKay1998} complemented with an extra pruning rule by the same authors and Exoo~\citep{Exoo2011}. Their algorithm receives as input $k,g$ and the order $n$. It starts from a $(k,g)$-Moore tree, adds $n-M(k,g)$ isolated vertices, and recursively adds edges one by one. To limit the search space, several pruning rules are applied. Additionally, a heuristic is applied to choose the next edge to add. In particular, the algorithm first adds the edge incident to the vertex $x$ that has the fewest available incident edges that can be added without violating the degree or girth constraint. As mentioned in~\citep{McKay1998}, this heuristic was one of the factors for the significant speedup compared to the previous best algorithm for $(k, g)$-graphs described in~\citep{Bri95}.

Our algorithm extends theirs with new optimisations to incorporate the diameter constraint. 
Contrary to their algorithm, we start from the tree $T$ of order $M'(k;g,d)$ corresponding to \cref{prop:imprLowerBound}. However, when $g=2t$ and $d=2t-1$, we start from the tree of girth $g-1$ (instead of $g$) since this tree contains more vertices, i.e., $M'(k;2t-1,2t-1) > M'(k;2t,2t-1)$, and is still a subgraph of any $(k;2t,2t-1)$-graph. After the construction of the tree, we proceed in a similar way: add isolated vertices until there are $n$ vertices and recursively add edges with an analogous heuristic that also incorporates the diameter. In addition to the pruning rules of~\citep{Exoo2011,McKay1998}, we also prune the search space by disallowing an edge that would decrease the distance between the two vertices $u,v$ at distance $d$ in $T$. All of this is described in more detailed pseudocode in Algorithm~\ref{alg:genGirthDiamGraphs} and~\ref{alg:backtrack}.

\begin{algorithm}[htb]
	\caption{\algoName{genGirthDiamGraphs}($n,k,g,d)$}\label{alg:genGirthDiamGraphs}
    \If{$n < M\kgd$ or $d < \floor{g/2}$}{
        \Comment{If the order $n$ is smaller than $M\kgd$ or if $d < \floor{g/2}$, there are definitely no graphs.}
        \Return
    }
    \Comment{Make tree corresponding to \cref{prop:imprLowerBound}.}
    $T \gets \algoName{makeTree}(n,k,g,d)$\\
    \Comment{Add $n-M\kgd$ isolated vertices to $T$.}
    $G_{start} \gets \algoName{addIsoVertices}(T, n-M\kgd)$\\
    \Comment{Determine which edges can be added without exceeding the  degree $k$ and without violating the girth $g$ and diameter $d$}
    $\validEdges \gets \algoName{calcValidAddableEdges}(G_{start}, k, g, d)$\\
    \Comment{Heuristic: edges will be added incident with the vertex $x$ that has the fewest options in $\validEdges$}
    $x \gets \arg\min_{w \in V(G), \degree{w}<r}(|\{e \colon e \in \validEdges \text{ and } w \in e \}|)$ \\
    $\algoName{\backtrack}(G_{start}, \validEdges, k, g, d, x)$ \Comment{See Algorithm~\ref{alg:backtrack}}
\end{algorithm}

The graph isomorphism detection is done with \texttt{nauty}~\citep{nauty} by computing and comparing the canonical form of two graphs. Here it is important to notice that $u$ and $v$ are two distinguished vertices in our algorithm (i.e., the pruning rules related to the diameter treat $u$ and $v$ different from other vertices in the graph). Therefore, one has to only consider isomorphisms between two graphs that map these distinguished vertices to each other. To achieve this, we compute the canonical form of the graph $G'$ that is obtained by adding two vertices $u',v'$ and two edges $\edge{u}{u'}, \edge{v}{v'}$ to the original graph $G$. This leads $u$ and $v$ to be of degree $k+1$. Since no other vertex can be of degree $k+1$, $u$ and $v$ will never be mapped to other vertices. The source code of our implementation of this algorithm can be found in the GitHub repository \url{https://github.com/AGT-Kulak/girthDiamGen}~\citep{githubRepo}.

\begin{algorithm}[htbp]
\DontPrintSemicolon
	\caption{\algoName{\backtrack}($G, \validEdges, k, g, d, x$)}\label{alg:backtrack}
	\Comment{$x$: current vertex to add edges to}
    \If{method was called with graph that is isomorphic with $G$ (where the isomorphism maps the two distinguished vertices whose distance is $d$ to each other)}{\Return}
        \If(\Comment*[h]{$G$ is a $k$-regular graph.}){$\forall v \in V(G) \colon \degree{v}=k$}{
            \If{ $g(G) = g$ and $diam(G)=d$}{
                output $G$ \\ 
		      }
        \Return \\
        }
        \Comment{Choose new $x$, if the current vertex $x$ has been completed to degree $k$.}
    	\If{$\degree{x}=k$}{
            $x \gets \arg\min_{w \in V(G), \degree{w}<k}(|\{e \colon e \in \validEdges \text{ and } w \in e \}|)$ \\
        }
    \Comment{Apply pruning rule from~\citep{Exoo2011} on all $\validEdges$.}
		$\validEdges \gets \algoName{pruneValidEdgesExoo}(G, \validEdges)$ \\
		\If{$\neg \textit{enoughValidEdgesLeft}(G, \validEdges)$}{\Return}
    \Comment{Try adding each valid edge incident to $x$.}    
    \ForEach{$w \in \{w \colon \edge{x}{w} \in \validEdges\}$}{
            \Comment{Copy $\validEdges$ in $prev\validEdges$ to restore it after removing $\edge{x}{w}$ again.}
            $prev\validEdges \gets \validEdges$\\
			$E(G) \gets E(G) \cup \{\edge{x}{w}\}$\\
            \Comment{Remove edges from $\validEdges$ resulting from adding edge $\edge{x}{w}$.}
            $\validEdges \gets \algoName{updateValidAddableEdges}(G, \validEdges, k, g, d)$\\
            \Comment{Check if we can prune $\edge{x}{w}$.}
            \If{none of the pruning rules can be applied}{ \label{lst:pruningChecks}
				$\algoName{\backtrack}(G, \validEdges, k, g, d, x)$ \\
			}
			$E(G) \gets E(G) \setminus \{\edge{x}{w}\}$\\
            \Comment{Restore $\validEdges$ and remove $\edge{x}{w}$ from it.}
            $\validEdges \gets prev\validEdges \setminus \{\edge{x}{w}\}$\\
			\If{$ \neg \textit{enoughValidEdgesLeft}(G, \validEdges)$}{
				\Return \\
			}
		%}
	}
\end{algorithm}

We ran the algorithm for various small values of $n,k,g,d$, leading to the determination of $n(k;g,d)$ and the corresponding cages. 
\cref{tab:gen0} and
\cref{tab:gen1,tab:gen2,tab:gen3,tab:gen4,tab:gen5} in Appendix~\ref{sec:tables_generator} show the results we obtained using our algorithm, i.e.\ the values of $n(k;g,d)$, the lower bound $M(k;g,d)$, (a lower bound on) the number of pairwise non-isomorphic $(k;g,d)$-cages and whether all generated $(k;g,d)$-cages are bipartite. In case $M(k;g,d)=n(k;g,d)$, both values are marked in bold. The generated $(k;g,d)$-cages can also be found in \texttt{graph6} format on the GitHub repository~\citep{githubRepo}.

\begin{table}[h]
\centering
\begin{tabular}{r r r | r r | l l }
\toprule
 $k$ & $g$ & $d$ & $M(k;g,d)$ & $n(k;g,d)$ & Number of cages & All bipartite \\ 
\midrule
3 & 6 & 3 & \textbf{14} & \textbf{14} & 1 & Yes \\
3 & 6 & 4 & \textbf{16} & \textbf{16} & 1 & Yes \\
3 & 6 & 5 & \textbf{20} & \textbf{20} & 6 & Yes \\
3 & 6 & 6 & \textbf{28} & \textbf{28} & 3016 & Yes \\
3 & 6 & 7 & 28 & 30 & 8 & No \\
3 & 6 & 8 & 29 & 32 & 7 & No \\
3 & 6 & 9 & 30 & 34 & 9 & No \\
3 & 6 & 10 & 32 & 36 & 6 & No \\
3 & 6 & 11 & 34 & 38 & 6 & No \\
3 & 6 & 12 & 38 & 44 & 13953 & No \\
3 & 6 & 13 & 42 & 44 & 1 & No \\
3 & 6 & 14 & 43 & 46 & 2 & No \\
3 & 6 & 15 & 44 & 48 & 6 & No \\
3 & 6 & 16 & 46 & 50 & 6 & No \\
3 & 6 & 17 & 48 & 52 & 6 & No \\
3 & 6 & 18 & 52 & 58 & 13987 & No \\
3 & 6 & 19 & 56 & 58 & 1 & No \\
3 & 6 & 20 & 57 & 60 & 2 & No \\
\bottomrule
\end{tabular}
\caption{Values of $M(k;g,d)$, $n(k;g,d)$, number of $(k;g,d)$-cages and bipartiteness for $(3;6,d)$ with $d\in \{3,\ldots,20\}$. In case $M(k;g,d)=n(k;g,d)$, both values are marked in bold.}
\label{tab:gen0}
\end{table}

Lastly, we explain which extra steps we took to ensure the correctness of the implementation of our
exhaustive generation algorithm. We used the generator \texttt{GENREG}~\citep{Meringer1999} to generate $(k,g)$-graphs of order $n$ (for small values of $n,k,g$) and split the graphs based on the diameter $d$, to obtain all $(k; g, d)$-graphs.
More specifically, we checked for $k=3$ and $g \in \{3, \ldots, 8\}$ and for $k \in \{4,5\}$ and $g \in \{3,4,5\}$ for orders $n(k,g), n(k,g)+1, \ldots, n(k,g)+C$, where $C$ was chosen such that \texttt{GENREG} was able to finish the computation within 100 CPU hours for the given $k,g$ and $n$. For each $n,k,g,d$, we obtained exactly the same set of $(k; g, d)$-graphs with our generator as \texttt{GENREG}. 

Besides that, we checked that we obtained all $(k,g)$-cages (for small $k,g$) with our generator by running it for the corresponding diameter of these $(k,g)$-cages. Moreover, Araujo-Pardo et al.\ found $(k;5,4)$-cages for $k \in \{3,4,5,6\}$~\citep{ACGKL25}. For each $k$, we obtained the same number of $(k;5,4)$-cages with our algorithm as Araujo-Pardo et al. Finally, we also verified that the results of our generator agree with the theoretical results for $(3;g,d)$ with $g \in \{4,5\}$ in \cref{sec:k3g4} (both for the order and number of girth-diameter cages).

\section{Further observations on $(k;g,d)$-cages}\label{sec:conclusion}

\subsection{Remark on $(k;3,3)$}

	Knor~\citep{Knor2014} considered the question of finding the minimum order of a $k$-regular graph of diameter $d$, which we will denote by $n(k;d)$. He determined $n(k;d)$ for all $k\geq2$ and $d\geq1$. He also describes $k$-regular graphs with diameter $d$ of order $n(k;d)$. For $k\geq3$ the described families have girth 3, besides for $d=3$, where the family has girth 4. Hence, we know $n(k;3,d)$  for all $d \neq3$. In \cref{prop:orderk33} we ``complete'' the determination of $n(k;3,d)$ by proving the order for $d=3$. 
	
    \begin{proposition}\label{prop:orderk33}
		 For all $k\geq3$, it holds that 
		 \begin{equation*}
		 	n(k;3,3)=2(k+1)
		  \end{equation*}
	\end{proposition}

    \begin{proof}
        Since the two vertices at distance $3$ have disjoint closed neighbourhoods, $n(k;3,3) \geq 2(k+1)$.
    
        Let $K_{k+1,k+1}$ be a complete bipartite graph with bipartition classes $U=\{u_0, u_1, \ldots, u_k\}$ and $V=\{v_0, v_1, \ldots, v_k\}$.
        Let $M=\{u_iv_i \colon 0 \le i \le k\}$ be a perfect matching.
        Let $G = K_{k+1,k+1} \setminus M \setminus \{u_1v_2, u_2v_1\} \cup \{u_1u_2, v_1v_2\}.$
        Now $\diam(G)=d(u_3,v_3)=3,$ while the girth is exactly three ($u_1u_2v_3$ is a triangle) and $G$ is $k$-regular. Hence $G$ certifies the equality $n(k;3,3) = 2(k+1)$.
    \end{proof}

\subsection{Behaviour of $(k;g,d)$-cages}\label{sec:behaviour}

Despite the existence of non-bipartite $(k;g,d)$-cages with even $g$, we can observe in \cref{tab:gen0} and \cref{tab:gen1,tab:gen2,tab:gen3,tab:gen4,tab:gen5} (in the Appendix) that all generated even girth $(k;g,d)$-cages with $d \le g$ are bipartite.
In other words, we only found non-bipartite $(k;g,d)$-cages for $d > g$. Often this is the case exactly when $G$ has a cutvertex (a regular bipartite graph cannot have a cutvertex), as is the case for the $(3; 6,12)$-cages, but e.g.\ the graph depicted in~\cref{fig:nonbip445cage} (also added as \url{https://houseofgraphs.org/graphs/54022} on the House of Graphs~\citep{HOG}) is a $2$-connected non-bipartite $(4; 4,5)$-cage.

\begin{figure}[h]
    \centering

    \begin{tikzpicture}[scale=0.6]

    \foreach \x in {0,10}{
        \foreach \y in {-1.5,1.5,0}{
        \draw[fill] (\x,\y) circle (0.25);
        }
    }

    \foreach \x in {-1.5,1.5,0.5,-0.5}{
    \foreach \y in {-1.5,1.5,0}{
    \draw (2,\x)--(0,\y);
    \draw (8,\x)--(10,\y);
    }
    }

\foreach \x in {2,8}{
        \foreach \y in {-1.5,1.5,0.5,-0.5}{
        \draw[fill] (\x,\y) circle (0.25);
        }
    }

\foreach \x in {1.5,0.5}{
        \draw (8,\x)--(6,1);
        \draw (2,\x)--(4,1);
        \draw (8,-\x)--(6,-1);
        \draw (2,-\x)--(4,-1);
        }

     \draw (6,-1)--(4,-1)--(4,1)--(6,1)--(6,-1);
    \foreach \x in {4,6}{
        \foreach \y in {1,-1}{
        \draw[fill] (\x,\y) circle (0.25);
        }
    }
    \end{tikzpicture} 

\caption{A $2$-connected non-bipartite $(4; 4,5)$-cage.}\label{fig:nonbip445cage}
\end{figure}
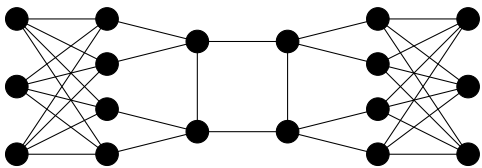

Moreover, $(k;g,d)$-cages with $d>g$ can never be $(k,g)$-cages since Sauer proved that every $(k,g)$-cage $G$ satisfies $\diam(G) \le g$~\citep{Sau1967}. 
Hence, if it holds that every $(k;g,d)$-cage with $g$ even and $d \le g$ is bipartite, then this proves the longstanding conjecture that every even girth $(k,g)$-cage is bipartite~\citep{EJ08,wong1982cages}.

Therefore, the following question can be of interest.

\begin{question}\label{ques:ques_onbip}
	For even $g$ and $d \le g$, is every $(k;g,d)$-cage bipartite?
\end{question}

Notice that~\cref{prop:imprLowerBound_specialimprovement} is sharp for the smallest case, $g=4$, by observing the extremal graphs $K_{k,k}$, $K_{k+1,k+1}\backslash M$ (a perfect matching $M$ removed), and the graph formed from two copies of $K_{k,k}^-$ ($K_{k,k}$ minus one edge) connected with edges between the degree $(k-1)$-vertices. Equality in the latter case can only be attained by a $k$-regular (connected) bipartite graph of order $4k$ for which there are two vertices $u,v$ in the same bipartition class with $N(u) \cap N(v) = \emptyset$.

Also for $g=6,8$ and $12$ one can conclude that the answer for Question~\ref{ques:ques_onbip} is positive for $d \in \{g/2,g\}$ whenever a $(k,g)$-Moore graph exists.

Further investigating $n\kgd$-graphs for $\floorfrac{g}{2} < d \le g$ may lead to improvements on the gap between the Moore bound and the known upper bound on the order of cages. In particular, a sharp bound on $n(k; g,g-1)$ might possibly be obtained by finding a good construction of a set of perfect matchings. The latter would already be interesting to think about for $g=10$.

\acknowledgements
\label{sec:ack}
%At the end of the manuscript, right before the bibliography you might
%want to place an acknowledgment. This can be easily done by using the 
%command \verb!\acknowledgements! as you can see here.
Several of the computational resources used in this work were provided by the VSC (Flemish Supercomputer Center), funded by the Research Foundation - Flanders (FWO) and the Flemish Government.

%\nocite{*}
%\bibliographystyle{abbrvnat}
% use the following instead if you encounter problems 
%\bibliographystyle{alpha}
%\bibliography{sample-dmtcs}

\bibliographystyle{abbrv}
\bibliography{ref}
\label{sec:biblio}

\newpage
\appendix

\section{Details on computer verification used in \cref{prop:main_not_sharp}}\phantomsection\label{app:comp_ver}

In this section we present some details of the computer verification we used in the proof of \cref{prop:main_not_sharp}. Let $G'$ be the graph induced by $N_{i+j} \cup \ldots \cup N_{i+j+7}$ (as presented in the proof). We were left to prove that there exists no $G'$ such that $|N_{i+j}| + \ldots + |N_{i+j+6}|=22$ and $|N_{i+j}| = |N_{i+j+7}| = \alpha \in \{2,3\}$. We obtained this result by making use of \texttt{multigraph}, an unpublished generator written by Gunnar Brinkmann, which is a generalisation of the generator for cubic graphs described in~\citep{Bri96}. The program \texttt{multigraph} generates all connected graphs of given minimum girth and given degree sequence (i.e., the amount of vertices of each degree). We want to generate graphs with $22+\alpha$ vertices. All vertices have degree $3$, except for the ones in $|N_{i+j}|$ and $|N_{i+j+7}|$. The ones in $N_{i+j}$ have degree $0, 1$ or $2$, and at least one in $N_{i+j}$ needs to have degree at least $1$.  The ones in $N_{i+j+7}$ have degree at most $3$, and at least one in $N_{i+j+7}$ has degree strictly less than $3$. We determined each degree sequence that satisfies the mentioned conditions (and such that the number of odd degree vertices is even) and were able to generate these graphs using \texttt{multigraph}. We then applied a simple filter program to the generated graphs that checks that each vertex of degree less than $3$ has at least one vertex at distance at least 7 (which is a necessary condition). Since none of the generated graphs passed this filter, we can conclude that no such $G'$ exists.

\section{Determined values of $n(k;g,d)$ with exhaustive generation}\phantomsection\label{sec:tables_generator}

\begin{table}[h]
\centering
\begin{tabular}{r r r | r r | l l }
\toprule
 $k$ & $g$ & $d$ & $M(k;g,d)$ & $n(k;g,d)$ & Number of cages & All bipartite \\ 
\midrule
3 & 4 & 2 & \textbf{6} & \textbf{6} & 1 & Yes \\
3 & 4 & 3 & \textbf{8} & \textbf{8} & 1 & Yes \\
3 & 4 & 4 & \textbf{12} & \textbf{12} & 4 & Yes \\
3 & 4 & 5 & 12 & 14 & 4 & No \\
3 & 4 & 6 & 13 & 16 & 7 & No \\
3 & 4 & 7 & 14 & 18 & 20 & No \\
3 & 4 & 8 & 16 & 20 & 38 & No \\
3 & 4 & 9 & 18 & 20 & 1 & No \\
3 & 4 & 10 & 19 & 22 & 4 & No \\
3 & 4 & 11 & 20 & 24 & 18 & No \\
3 & 4 & 12 & 22 & 26 & 40 & No \\
3 & 4 & 13 & 24 & 26 & 1 & No \\
\bottomrule
\end{tabular}
\caption{Values of $M(k;g,d)$, $n(k;g,d)$, number of $(k;g,d)$-cages and bipartiteness for $(3;4,d)$ with $d\in \{2,\ldots,13\}$. In case $M(k;g,d)=n(k;g,d)$, both values are marked in bold.}
\label{tab:gen1}
\end{table}

\begin{table}[h]
\centering
\begin{tabular}{r r r | r r | l l }
\toprule
 $k$ & $g$ & $d$ & $M(k;g,d)$ & $n(k;g,d)$ & Number of cages & All bipartite \\ 
\midrule
3 & 5 & 2 & \textbf{10} & \textbf{10} & 1 &  \\
3 & 5 & 3 & 11 & 12 & 2 &  \\
3 & 5 & 4 & \textbf{14} & \textbf{14} & 2 &  \\
3 & 5 & 5 & \textbf{20} & \textbf{20} & 470 &  \\
3 & 5 & 6 & 21 & 22 & 6 &  \\
3 & 5 & 7 & \textbf{22} & \textbf{22} & 1 &  \\
3 & 5 & 8 & \textbf{24} & \textbf{24} & 4 &  \\
3 & 5 & 9 & \textbf{26} & \textbf{26} & 10 &  \\
3 & 5 & 10 & \textbf{30} & \textbf{30} & 241 &  \\
3 & 5 & 11 & 31 & 32 & 10 &  \\
3 & 5 & 12 & \textbf{32} & \textbf{32} & 1 &  \\
3 & 5 & 13 & \textbf{34} & \textbf{34} & 4 &  \\
3 & 5 & 14 & \textbf{36} & \textbf{36} & 10 &  \\
3 & 5 & 15 & \textbf{40} & \textbf{40} & 308 &  \\
3 & 5 & 16 & 41 & 42 & 15 &  \\
3 & 5 & 17 & \textbf{42} & \textbf{42} & 1 &  \\
3 & 5 & 18 & \textbf{44} & \textbf{44} & 4 &  \\
3 & 5 & 19 & \textbf{46} & \textbf{46} & 10 &  \\
3 & 5 & 20 & \textbf{50} & \textbf{50} & 354 &  \\
3 & 5 & 21 & 51 & 52 & 19 &  \\
3 & 5 & 22 & \textbf{52} & \textbf{52} & 1 &  \\
3 & 5 & 23 & \textbf{54} & \textbf{54} & 4 &  \\
3 & 5 & 24 & \textbf{56} & \textbf{56} & 10 &  \\
3 & 5 & 25 & \textbf{60} & \textbf{60} & 421 &  \\
3 & 5 & 26 & 61 & 62 & 24 &  \\
3 & 5 & 27 & \textbf{62} & \textbf{62} & 1 &  \\
3 & 5 & 28 & \textbf{64} & \textbf{64} & 4 &  \\
3 & 5 & 29 & \textbf{66} & \textbf{66} & 10 &  \\
3 & 5 & 30 & \textbf{70} & \textbf{70} & 467 &  \\
3 & 5 & 31 & 71 & 72 & 28 &  \\
3 & 5 & 32 & \textbf{72} & \textbf{72} & 1 &  \\
3 & 5 & 33 & \textbf{74} & \textbf{74} & 4 &  \\
3 & 5 & 34 & \textbf{76} & \textbf{76} & 10 &  \\
3 & 5 & 35 & \textbf{80} & \textbf{80} & 534 &  \\
3 & 5 & 36 & 81 & 82 & 33 &  \\
3 & 5 & 37 & \textbf{82} & \textbf{82} & 1 &  \\
3 & 5 & 38 & \textbf{84} & \textbf{84} & 4 &  \\
3 & 5 & 39 & \textbf{86} & \textbf{86} & 10 &  \\
3 & 5 & 40 & \textbf{90} & \textbf{90} & 580 &  \\
\hline
3 & 6 & 3 & \textbf{14} & \textbf{14} & 1 & Yes \\
3 & 6 & 4 & \textbf{16} & \textbf{16} & 1 & Yes \\
3 & 6 & 5 & \textbf{20} & \textbf{20} & 6 & Yes \\
3 & 6 & 6 & \textbf{28} & \textbf{28} & 3016 & Yes \\
3 & 6 & 7 & 28 & 30 & 8 & No \\
3 & 6 & 8 & 29 & 32 & 7 & No \\
\bottomrule
\end{tabular}
\caption{Values of $M(k;g,d)$, $n(k;g,d)$, number of $(k;g,d)$-cages and bipartiteness for $(3;5,d)$ with $d\in \{2,\ldots,40\}$ and $(3;6,d)$ with $d\in \{3,\ldots,8\}$. In case $M(k;g,d)=n(k;g,d)$, both values are marked in bold.}
\label{tab:gen2}
\end{table}

\begin{table}[h]
\centering
\begin{tabular}{r r r | r r | l l }
\toprule
 $k$ & $g$ & $d$ & $M(k;g,d)$ & $n(k;g,d)$ & Number of cages & All bipartite \\ 
\midrule
3 & 6 & 9 & 30 & 34 & 9 & No \\
3 & 6 & 10 & 32 & 36 & 6 & No \\
3 & 6 & 11 & 34 & 38 & 6 & No \\
3 & 6 & 12 & 38 & 44 & 13953 & No \\
3 & 6 & 13 & 42 & 44 & 1 & No \\
3 & 6 & 14 & 43 & 46 & 2 & No \\
3 & 6 & 15 & 44 & 48 & 6 & No \\
3 & 6 & 16 & 46 & 50 & 6 & No \\
3 & 6 & 17 & 48 & 52 & 6 & No \\
3 & 6 & 18 & 52 & 58 & 13987 & No \\
3 & 6 & 19 & 56 & 58 & 1 & No \\
3 & 6 & 20 & 57 & 60 & 2 & No \\
\hline
3 & 7 & 4 & 23 & 24 & 1 &  \\
3 & 7 & 5 & \textbf{26} & \textbf{26} & 1 &  \\
3 & 7 & 6 & \textbf{32} & \textbf{32} & 19 &  \\
3 & 7 & 7 & \textbf{44} & \textbf{44} & $\geq$ 1 &  \\
3 & 7 & 8 & 45 & 48 & $\geq$ 1 &  \\
3 & 7 & 9 & 46 & 48 & $\geq$ 1 &  \\
3 & 7 & 10 & 48 & 50 & $\geq$ 1 &  \\
3 & 7 & 11 & \textbf{50} & \textbf{50} & $\geq$ 1 &  \\
3 & 7 & 12 & \textbf{54} & \textbf{54} & $\geq$ 1 &  \\
3 & 7 & 13 & \textbf{58} & \textbf{58} & $\geq$ 1 &  \\
3 & 7 & 14 & \textbf{66} & \textbf{66} & $\geq$ 1 &  \\
3 & 7 & 15 & 67 & 72 & $\geq$ 1 &  \\
3 & 7 & 16 & 68 & 72 & $\geq$ 1 &  \\
3 & 7 & 17 & 70 & 74 & $\geq$ 1 &  \\
3 & 7 & 18 & 72 & 74 & $\geq$ 1 &  \\
3 & 7 & 19 & 76 & 78 & $\geq$ 1 &  \\
3 & 7 & 20 & 80 & 82 & $\geq$ 1 &  \\
3 & 7 & 21 & \textbf{88} & \textbf{88} & $\geq$ 1 &  \\
3 & 7 & 22 & 89 & 94 & $\geq$ 1 &  \\
3 & 7 & 23 & 90 & 96 & $\geq$ 1 &  \\
3 & 7 & 24 & 92 & 98 & $\geq$ 1 &  \\
3 & 7 & 25 & 94 & 98 & $\geq$ 1 &  \\
3 & 7 & 26 & 98 & 102 & $\geq$ 1 &  \\
3 & 7 & 27 & 102 & 106 & $\geq$ 1 &  \\
3 & 7 & 28 & 110 & 112 & $\geq$ 1 &  \\
3 & 7 & 29 & 111 & 118 & $\geq$ 1 &  \\
3 & 7 & 30 & 112 & 120 & $\geq$ 1 &  \\
3 & 7 & 31 & 114 & 122 & $\geq$ 1 &  \\
3 & 7 & 32 & 116 & 122 & $\geq$ 1 &  \\
3 & 7 & 33 & 120 & 126 & $\geq$ 1 &  \\
3 & 7 & 34 & 124 & 130 & $\geq$ 1 &  \\
3 & 7 & 35 & 132 & 136 & $\geq$ 1 &  \\
\hline
3 & 8 & 4 & \textbf{30} & \textbf{30} & 1 & Yes \\
\bottomrule
\end{tabular}
\caption{Values of $M(k;g,d)$, $n(k;g,d)$, number of $(k;g,d)$-cages and bipartiteness for $(3;6,d)$ with $d\in \{9,\ldots,20\}$, $(3;7,d)$ with $d\in \{4,\ldots,35\}$ and $(3;8,d)$ with $d = 4$. In case $M(k;g,d)=n(k;g,d)$, both values are marked in bold.}
\label{tab:gen3}
\end{table}

\begin{table}[h]
\centering
\begin{tabular}{r r r | r r | l l }
\toprule
 $k$ & $g$ & $d$ & $M(k;g,d)$ & $n(k;g,d)$ & Number of cages & All bipartite \\ 
\midrule
3 & 8 & 5 & 32 & 34 & 1 & Yes \\
3 & 8 & 6 & \textbf{36} & \textbf{36} & 2 & Yes \\
3 & 8 & 7 & \textbf{44} & \textbf{44} & 802 & Yes \\
3 & 8 & 8 & \textbf{60} & \textbf{60} & $\geq$ 7170730 & All found graphs are bipartite \\
3 & 8 & 9 & 60 & 62 & 4 & No \\
3 & 8 & 10 & 61 & 64 & 1 & Yes \\
\hline
3 & 10 & 6 & 64 & 70 & 3 & Yes \\
3 & 10 & 7 & 68 & 72 & 1 & Yes \\
3 & 10 & 8 & 76 & 80 & 5 & Yes \\
3 & 10 & 9 & \textbf{92} & \textbf{92} & $\geq$ 19495 & All found graphs are bipartite \\
\hline
4 & 4 & 2 & \textbf{8} & \textbf{8} & 1 & Yes \\
4 & 4 & 3 & \textbf{10} & \textbf{10} & 1 & Yes \\
4 & 4 & 4 & \textbf{16} & \textbf{16} & 102 & Yes \\
4 & 4 & 5 & 16 & 18 & 10 & No \\
4 & 4 & 6 & 17 & 20 & 10 & No \\
4 & 4 & 7 & 18 & 22 & 7 & No \\
4 & 4 & 8 & 21 & 24 & 6 & Yes \\
4 & 4 & 9 & 24 & 26 & 6 & Yes \\
4 & 4 & 10 & 25 & 28 & 6 & Yes \\
4 & 4 & 11 & 26 & 30 & 6 & Yes \\
\hline
4 & 5 & 3 & 18 & 19 & 1 &  \\
4 & 5 & 4 & \textbf{22} & \textbf{22} & 4 &  \\
4 & 5 & 5 & \textbf{34} & \textbf{34} & $\geq$ 1 &  \\
4 & 5 & 6 & 35 & 37 & $\geq$ 1 &  \\
4 & 5 & 7 & 36 & 38 & $\geq$ 1 &  \\
4 & 5 & 8 & \textbf{39} & \textbf{39} & $\geq$ 1 &  \\
4 & 5 & 9 & 42 & 45 & $\geq$ 1 &  \\
4 & 5 & 10 & \textbf{51} & \textbf{51} & $\geq$ 1 &  \\
4 & 5 & 11 & 52 & 55 & $\geq$ 1 &  \\
4 & 5 & 12 & 53 & 57 & $\geq$ 1 &  \\
4 & 5 & 13 & 56 & 58 & $\geq$ 1 &  \\
4 & 5 & 14 & 59 & 62 & $\geq$ 1 &  \\
4 & 5 & 15 & \textbf{68} & \textbf{68} & $\geq$ 1 &  \\
4 & 5 & 16 & 69 & 73 & $\geq$ 1 &  \\
4 & 5 & 17 & 70 & 76 & $\geq$ 1 &  \\
4 & 5 & 18 & 73 & 77 & $\geq$ 1 &  \\
4 & 5 & 19 & 76 & 81 & $\geq$ 1 &  \\
4 & 5 & 20 & \textbf{85} & \textbf{85} & $\geq$ 1 &  \\
4 & 5 & 21 & 86 & 91 & $\geq$ 1 &  \\
4 & 5 & 22 & 87 & 95 & $\geq$ 1 &  \\
4 & 5 & 23 & 90 & 96 & $\geq$ 1 &  \\
\hline
4 & 6 & 3 & \textbf{26} & \textbf{26} & 1 & Yes \\
4 & 6 & 4 & \textbf{28} & \textbf{28} & 1 & Yes \\
4 & 6 & 5 & \textbf{34} & \textbf{34} & 146 & Yes \\
4 & 6 & 6 & \textbf{52} & \textbf{52} & $\geq$ 396779 & All found graphs are bipartite \\
\bottomrule
\end{tabular}
\caption{Values of $M(k;g,d)$, $n(k;g,d)$, number of $(k;g,d)$-cages and bipartiteness for $(3;8,d)$ with $d\in \{5,\ldots,10\}$, $(3;10,d)$ with $d\in \{6,\ldots,9\}$, $(4;4,d)$ with $d\in \{2,\ldots,11\}$, $(4;5,d)$ with $d\in \{3,\ldots,23\}$ and $(4;6,d)$ with $d\in \{3,\ldots,6\}$. In case $M(k;g,d)=n(k;g,d)$, both values are marked in bold.}
\label{tab:gen4}
\end{table}

\begin{table}[h]
\centering
\begin{tabular}{r r r | r r | l l }
\toprule
 $k$ & $g$ & $d$ & $M(k;g,d)$ & $n(k;g,d)$ & Number of cages & All bipartite \\ 
\midrule
4 & 6 & 7 & 52 & 54 & 14 & Yes \\
4 & 6 & 8 & 53 & 56 & 4 & Yes \\
4 & 6 & 9 & 54 & 60 & 272 & Yes \\
\hline
4 & 8 & 4 & \textbf{80} & \textbf{80} & 1 & Yes \\
\hline
5 & 4 & 2 & \textbf{10} & \textbf{10} & 1 & Yes \\
5 & 4 & 3 & \textbf{12} & \textbf{12} & 1 & Yes \\
5 & 4 & 4 & \textbf{20} & \textbf{20} & 176581 & Yes \\
5 & 4 & 5 & 20 & 22 & 19 & Yes \\
5 & 4 & 6 & 21 & 24 & 9 & Yes \\
5 & 4 & 7 & 22 & 28 & 537 & No \\
5 & 4 & 8 & 26 & 32 & 25408 & No \\
\hline
5 & 5 & 3 & 27 & 30 & 4 &  \\
5 & 5 & 4 & \textbf{32} & \textbf{32} & 7 &  \\
\hline
5 & 6 & 3 & \textbf{42} & \textbf{42} & 1 & Yes \\
5 & 6 & 4 & 44 & 46 & 1 & Yes \\
5 & 6 & 5 & \textbf{52} & \textbf{52} & $\geq$ 440 & All found graphs are bipartite \\
\hline
6 & 4 & 2 & \textbf{12} & \textbf{12} & 1 & Yes \\
6 & 4 & 3 & \textbf{14} & \textbf{14} & 1 & Yes \\
6 & 4 & 4 & \textbf{24} & \textbf{24} & $\geq$ 20343552 & All found graphs are bipartite \\
6 & 4 & 5 & 24 & 26 & 49 & Yes \\
6 & 4 & 6 & 25 & 28 & 19 & Yes \\
6 & 4 & 7 & 26 & 32 & 1184 & Yes \\
\hline
6 & 5 & 3 & 38 & 40 & 1 &  \\
6 & 5 & 4 & \textbf{44} & \textbf{44} & 2 &  \\
\hline
6 & 6 & 3 & \textbf{62} & \textbf{62} & 1 & Yes \\
\hline
7 & 4 & 2 & \textbf{14} & \textbf{14} & 1 & Yes \\
7 & 4 & 3 & \textbf{16} & \textbf{16} & 1 & Yes \\
7 & 4 & 4 & \textbf{28} & \textbf{28} & $\geq$ 2987496 & All found graphs are bipartite \\
7 & 4 & 5 & 28 & 30 & 113 & Yes \\
\hline
7 & 5 & 2 & \textbf{50} & \textbf{50} & 1 &  \\
\bottomrule
\end{tabular}
\caption{Values of $M(k;g,d)$, $n(k;g,d)$, number of $(k;g,d)$-cages and bipartiteness for $(4;6,d)$ with $d\in \{7,8,9\}$, $(4;8,d)$ with $d = 4$, $(5;4,d)$ with $d\in \{2,\ldots,8\}$, $(5;5,d)$ with $d\in \{3,4\}$, $(5;6,d)$ with $d\in \{3,4,5\}$, $(6;4,d)$ with $d\in \{2,\ldots,7\}$, $(6;5,d)$ with $d\in \{3,4\}$, $(6;6,d)$ with $d = 3$, $(7;4,d)$ with $d\in \{2,\ldots,5\}$ and $(7;5,d)$ with $d = 2$. In case $M(k;g,d)=n(k;g,d)$, both values are marked in bold.}
\label{tab:gen5}
\end{table}

\end{document}